\pdfoutput=1

\documentclass[10pt,reqno,a4paper,dvipsnames,x11names]{amsart}
\usepackage[utf8]{inputenc}
\usepackage[T1]{fontenc}
\usepackage[british]{babel}
\usepackage[a4paper,margin=1in]{geometry}
\usepackage{microtype}
\usepackage{mathpazo}
\usepackage{amsmath,amssymb,amsthm,mathtools,mathrsfs,bm,eucal,tensor}
\usepackage{stmaryrd}
\usepackage{dsfont}
\usepackage[all,cmtip]{xy}
\usepackage{tikz}
\usetikzlibrary{
	arrows.meta,
	calc,
	positioning,
	fit,
	backgrounds,
	shapes.geometric,
	decorations.markings
}
\usepackage{tikz-cd}

\tikzset{
	>=Stealth,
	box/.style={draw, rounded corners, inner sep=6pt},
	v/.style={draw, circle, inner sep=1.4pt},
	arr/.style={->, thick},
	darr/.style={->, thick, dashed},
	midarrow/.style={
		postaction={decorate},
		decoration={markings,mark=at position 0.58 with {\arrow{Stealth}}}
	}
}
\usepackage{enumitem}
\usepackage{comment,braket,xspace,csquotes}
\usepackage{xcolor}
\usepackage[hidelinks]{hyperref}
\usepackage[capitalize,nameinlink]{cleveref}

\numberwithin{equation}{section}
\theoremstyle{plain}
\newtheorem{theorem}{Theorem}[section]
\newtheorem{proposition}[theorem]{Proposition}
\newtheorem{corollary}[theorem]{Corollary}
\newtheorem{lemma}[theorem]{Lemma}

\theoremstyle{definition}
\newtheorem{definition}[theorem]{Definition}

\theoremstyle{remark}
\newtheorem{remark}[theorem]{Remark}
\newcommand{\C}{\mathbb{C}}
\newcommand{\R}{\mathbb{R}}
\newcommand{\Z}{\mathbb{Z}}

\newcommand{\Q}{\mathbb{Q}}
\newcommand{\PP}{\mathbb{P}}

\newcommand{\cO}{\mathcal{O}}

\newcommand{\fol}{\mathcal{F}}

\newcommand{\Aff}{\operatorname{Aff}}
\newcommand{\im}{\mathrm{im}}
\newcommand{\End}{\mathrm{End}}
\newcommand{\Aut}{\mathrm{Aut}}
\newcommand{\Hom}{\mathrm{Hom}}
\newcommand{\rk}{\mathrm{rk}}

\newcommand{\PGL}{\mathrm{PGL}}
\newcommand{\Div}{\mathrm{Div}}
\newcommand{\Res}{\mathrm{Res}}

\newcommand{\Sone}{S^1}
\title{Stacky geometry and logarithmic topology of transversely affine foliations}

\author[P. P. Barbassa]{Pedro P. Barbassa}
\address[P. P. Barbassa]{IMECC - UNICAMP \\ Departamento de Matem\'atica \\
	Rua S\'ergio Buarque de Holanda, 651\\ 13083-970 Campinas-SP, Brazil}
\email{p185839@dac.unicamp.br}

\author[G. S. Barbosa]{Gabriel Santos Barbosa}
\address[G. S. Barbosa]{Dipartimento di Matematica, Universit\`a degli Studi di Bari, Via E. Orabona 4, I-70125, Bari, Italy}
\email{gabrielbarbosa.math@gmail.com}

\author[M. Corr\^ea]{Maur\'icio Corr\^ea}
\address[M. Corr\^ea]{Dipartimento di Matematica, Universit\`a degli Studi di Bari, Via E. Orabona 4, I-70125, Bari, Italy}
\email{mauricio.correa.mat@gmail.com, mauricio.barros@uniba.it}

\begin{document}
	
	\begin{abstract}
		We study the linear monodromy of codimension-one transversely affine
		foliations through multiplicative quotient stacks and logarithmic
		boundary topology. When the transverse affine structure admits a global
		presentation on $U$, a choice of multiplicative developing coordinate
		induces a morphism to $[\C^*/\Gamma]$, where $\Gamma\subset\C^*$ is the
		image of the linear monodromy. We prove that for infinite
		$\Gamma$ every strict semi-equivariant holomorphic self-map of $\C^*$
		is a monomial, and classify the nonconstant strict semi-equivariant
		rational self-maps of $\PP^1$ for arbitrary $\Gamma$. For a logarithmic
		connection along a simple normal crossings divisor, we describe the
		meridional characters and angular foliations on the boundary fibres of
		the Kato--Nakayama space, including the topology and closures of their
		leaves. We also describe the extension of the linear monodromy local
		system to this space and the effect of local Kummer base change on
		finite meridional monodromy.
	\end{abstract}
	
	\maketitle
	
	\section{Introduction}\label{sec:intro}
	
	The study of Liouvillian first integrals of planar polynomial
	differential equations leads to integrating factors obtained by
	exponentiating primitives of closed rational $1$-forms, as in Singer's
	theorem \cite{Singer92}. Sc\'ardua studied transversely affine and
	projective structures on singular holomorphic foliations, including
	their behaviour along invariant hypersurfaces \cite{Scardua97}.
	For a codimension-one foliation, a transverse affine structure can
	be described by a flat meromorphic connection on its normal bundle
	\cite{CousinPereira,LPT14}. The affine monodromy of
	the foliation and the linear monodromy of this connection are related,
	but retain different information. We examine the multiplicative
	quotient associated with the linear part and its behaviour near a
	logarithmic boundary.
	
	To describe the distinction, suppose that $\fol$ admits a global
	presentation $(\omega,\eta)$ on a connected open set $U$, with
	$d\omega=\omega\wedge\eta$ and $d\eta=0$. On the universal cover
	$p:\widetilde U\to U$, choose a primitive $\Phi$ of $p^*\eta$. Then
	\[
	z=e^\Phi,\qquad \frac{dz}{z}=p^*\eta,
	\qquad df=e^\Phi p^*\omega
	\]
	for a holomorphic function $f$ on $\widetilde U$. The function $f$ is
	an affine developing map of $\fol$, whereas $z$ is a multiplicative
	developing coordinate for the linear connection. Under deck
	transformations, $f$ changes by affine transformations and $z$ by
	multiplication. In particular, the latter is not generally a first
	integral of $\fol$.
	
	Let $\xi:\pi_1(U)\to\Aff(\C)$ be the affine monodromy representation,
	let $\rho:\pi_1(U)\to\C^*$ be its linear part and put
	$\Gamma=\operatorname{im}(\rho)$. We regard $\Gamma$ as a discrete
	group, even when its image is not discrete in $\C^*$. The character
	$\rho$ determines a principal $\Gamma$-bundle over $U$; the chosen
	coordinate $z$ equips this bundle with an equivariant map to $\C^*$,
	and hence defines a morphism
	$U\to[\C^*/\Gamma]$. By contrast, the translational part of $\xi$
	determines a flat affine torsor, which is not determined by $\rho$
	alone. The relation between these constructions is made precise in
	\Cref{prop:translational-class-and-affine-torsor-bridge,cor:affine-linear-bridge}.
	
	We first study maps of the fixed multiplicative presentation. Such a
	map is given by a holomorphic function
	$\widetilde\psi:\C^*\to\C^*$ and a homomorphism
	$\sigma:\Gamma\to\Gamma$ satisfying
	$\widetilde\psi(\gamma z)=\sigma(\gamma)\widetilde\psi(z)$. If
	$\Gamma$ is infinite, we prove that every such map has the form
	$\widetilde\psi(z)=az^n$, where $a\in\C^*$ and $n\in\Z$, and that
	$\sigma(\gamma)=\gamma^n$ (\Cref{proposition:hol-rigidity}). For
	rational maps of $\PP^1$, the classification has two cases: the
	nonconstant maps for infinite $\Gamma$ are precisely the monomials
	$az^n$ with $n\ne0$, while for $\Gamma=\mu_N$ they have the form
	\[
	R(z)=z^m\frac{P(z^N)}{Q(z^N)},\qquad
	\sigma(\zeta)=\zeta^m,
	\]
	where $P,Q\in\C[t]$ are coprime and the resulting rational map is
	nonconstant (\Cref{thm:EndP1-regimes}). These are classifications of
	strict semi-equivariant maps, not of arbitrary morphisms of the
	quotient stack.
	
	The distinction matters even for a properly discontinuous action.
	If $\Gamma=q^\Z$ with $|q|\ne1$, then $\C^*/q^\Z$ is an elliptic
	curve with period lattice
	$\Lambda=2\pi i\Z+(\log q)\Z$. Its holomorphic self-maps lift to
	affine maps on $\C$ with linear coefficient $\alpha$ satisfying
	$\alpha\Lambda\subset\Lambda$. A map represented on the chosen
	multiplicative atlas must instead have integral exponent. Thus
	complex multiplication can give endomorphisms of the elliptic curve
	without strict representatives on that atlas
	(\Cref{proposition:End-elliptic,prop:qZ-rigidity}). This also
	explains why descent through a foliation groupoid and morphisms of
	the multiplicative quotient must be treated separately.
	
	Logarithmic structures also arise naturally in the study of
	foliations on normal crossings varieties and their deformations
	\cite{CorreaPerrellaVelazquez26}. For a transversely affine
	foliation, the linear monodromy has a local description in terms
	of residues. Let $D$ be a simple normal crossings divisor on $X$
	and suppose that $\eta$ extends as a closed logarithmic
	$1$-form. At a
	point of $D_I^\circ$, write
	$\alpha_i=\operatorname{Res}_{D_i}(\eta)$ for $i\in I$. The
	Kato--Nakayama space $X^{\log}$ is locally the real oriented blow-up
	of $D$; its fibre over this point is a real torus, noncanonically
	isomorphic to $(S^1)^{|I|}$ \cite{KN99}. The character of its
	meridional fundamental group is
	\[
	(m_i)_{i\in I}\longmapsto
	\exp\!\left(2\pi i\sum_{i\in I}m_i\alpha_i\right).
	\]
	In angular coordinates the same residues determine a
	constant-coefficient distribution with tangent space
	$V=\ker(a)\cap\ker(b)\subset\R^{|I|}$, where
	$a=(\Re\alpha_i)$ and $b=(\Im\alpha_i)$. We describe its leaves as
	immersed products of Euclidean spaces and real tori, and determine
	their closures from the rational subspace generated by $V$
	(\Cref{thm:topology-angular-leaf}). These angular foliations on the
	boundary fibres are distinct from the real distribution defined by
	$\eta$ on the complement of $D$.
	
	The inclusion $U=X\setminus D\hookrightarrow X^{\log}$ is a homotopy
	equivalence in the normal crossings case. Consequently the rank-one
	local system determined by $\rho$ extends to $X^{\log}$, where its
	restriction to a boundary fibre has the meridional character above
	(\Cref{thm:Flog}). This does not require
	the developing coordinate to extend across $D$. When the meridional
	character has finite image, a suitable local Kummer base change
	trivialises it, although the angular information remains on the
	Kato--Nakayama boundary before the base change
	(\Cref{thm:kummer-vs-xlog-boundary,cor:kummer-trivialisation}).
	
	\section{Transverse affine structures and linear monodromy}\label{sec:holonomy}
	
	We recall the flat-connection description of transverse affine structures and the relation between their linear monodromy, periods and residues.

	\subsection{Flat connections and transverse affine structures}\label{subsec:connection-normal}
	
	We use the description of a transverse affine structure by a flat meromorphic connection on the normal bundle, following Cousin--Pereira \cite{CousinPereira}.
	Let $X$ be a complex manifold and let $\fol$ be a codimension one holomorphic foliation on $X$ with normal bundle $N_{\fol}$.
	We represent $\fol$ by a section
	$
	\omega\in H^0\bigl(X,\,N_{\fol}\otimes \Omega_X^1\bigr)
	$
	whose zero locus has codimension at least $2$ and satisfying $\omega\wedge d\omega=0$.
	
	\begin{definition}\label[definition]{def:TA-connection}
		A \emph{singular transverse affine structure} for $\fol$ is a flat meromorphic connection
		\[
		\nabla: N_{\fol}\longrightarrow N_{\fol}\otimes \Omega_X^1(*D)
		\]
		along a reduced divisor $D\subset X$ such that $\nabla(\omega)=0$.
		We say that $D$ is \emph{minimal} if the connection form of $\nabla$ is not holomorphic at any point of $D$.
	\end{definition}
	
	On an open set $U\subset X$ where $N_{\fol}$ is trivial, choose a
	nowhere-vanishing frame $e$. The connection is given by a meromorphic
	$1$-form $\eta_0\in\Omega_X^1(*D)(U)$ through
	\[
	\nabla(fe)=e\otimes(df+f\eta_0).
	\]
	Writing $\omega$ in the same trivialisation as a meromorphic $1$-form $\omega_0$ on $U$, the equation $\nabla(\omega)=0$ becomes
	\[
	d\omega_0=\omega_0\wedge \eta_0,\qquad d\eta_0=0,
	\]
	and changing trivialisation by a nowhere vanishing function $g\in \cO_X(U)^*$ sends
	\[
	(\omega_0,\eta_0)\longmapsto (g\,\omega_0,\ \eta_0-d\log g).
	\]
	The same change of frame will be used throughout. Whenever a single form
	$\eta$ is used for period integrals or developing functions on $U=X\setminus D$,
	we assume that the connection has been given in a fixed global
	trivialisation of $N_{\fol}|_U$. In the absence of such a trivialisation,
	its linear character is defined intrinsically by parallel transport;
	the period formula in terms of one global $1$-form is not available.
	
	\subsection{Affine monodromy and its linear part}\label{subsec:affine-monodromy}
	
	Let $U=X\setminus D$, and assume that the transverse affine connection is holomorphic on $U$. Fix a base point $q\in U$. By \cite[\S2]{CousinPereira}, the transverse affine structure determines a monodromy representation
	$
	\xi:\pi_1(U,q)\longrightarrow \Aff(\C)=\C^*\ltimes\C,
	$
	whose linear part is the multiplicative character
	$\rho:\pi_1(U,q)\longrightarrow \C^*$. In a fixed global trivialisation
	of the normal bundle over $U$, with connection form $\eta_0$, one has
	$
	\rho(\gamma)=\exp\!\left(\int_\gamma \eta_0\right).
	$
	With this convention, $\rho$ is the monodromy character of the dual of the flat connection $\nabla=d+\eta_0$. 
	Choose a lift $\tilde q$ of $q$ on the universal cover $\widetilde U$ and normalise
	$\Phi(\tilde q)=0$. If $f(\tilde q)=0$ and
	$df=e^\Phi p^*\omega_0$, then
	\[
	\xi(\gamma)(w)=\rho(\gamma)w+\tau(\gamma),\qquad
	\tau(\gamma)=\int_{\tilde q}^{\gamma\tilde q}e^\Phi p^*\omega_0.
	\]
	The integral is independent of the path in $\widetilde U$, and the
	affine cocycle identity follows from the transformation law for $f$.
	The affine representation projects to its linear character under
	$
	\Aff(\C)\to\C^*,
	$
	with $\rho$ as its linear part:
	\begin{center}
		\begin{tikzcd}[column sep=large,row sep=large]
			\pi_1(U,q) \arrow[r, "\xi"] \arrow[dr, "\rho"'] & \Aff(\C)=\C^*\ltimes \C \arrow[d] \\
			& \C^*
		\end{tikzcd}
	\end{center}
	
	The multiplicative quotient considered below depends on the linear character alone. Writing
	$
	\Gamma:=\im(\rho)\subset\C^*,
	$
	we consider
	$
	[\C^*/\Gamma]
	$
	and its compactification
	$[\PP^1/\Gamma].
	$
	The translational part of $\xi$ defines a $1$-cocycle with values in the $\rho$-twisted additive local system. The multiplicative coordinate records the linear character, whereas first integrals of $\fol$ arise from the affine developing map. The two constructions will be kept distinct.
	\begin{remark}
		If $\eta=0$ in a global presentation, then $\rho$ is trivial and
		the multiplicative quotient is $\C^*$ (or $\PP^1$ after
		compactification). The affine monodromy may still act by translations
		on the affine developing coordinate.
	\end{remark}
	
	\subsection{Transversely projective structures and their monodromy}\label{subsec:projective-structure}

	We recall the standard projective formalism (see \cite[\S2]{LPT14} and the references therein) in a form compatible with the stack-theoretic viewpoint developed later in this article.
	
	\begin{definition} \label[definition]{def:TP}
		Let $X$ be a complex manifold and let $\fol$ be a codimension one foliation.
		A (singular) \emph{transversely projective structure} on $(X,\fol)$ along a reduced divisor $D$ is a rank two vector bundle $E$ on $X$
		together with a flat meromorphic connection
		\[
		\nabla_E:E\longrightarrow E\otimes \Omega_X^1(*D),
		\]
		and a (meromorphic) line subbundle $L\subset E$ on $X\setminus Z$ (for some analytic subset $Z$ of codimension at least $2$)
		such that the meromorphic section $\PP(L)$ is generically transverse to the induced Riccati foliation and its pullback defines $\fol$ on $X\setminus Z$.
		Equivalently (in a local trivialisation), one may encode the projective structure by a triple of meromorphic $1$-forms
		$(\omega_0,\omega_1,\omega_2)$ satisfying the structure equations
		\[
		d\omega_0=\omega_0\wedge \omega_1,\qquad
		d\omega_1=2\,\omega_0\wedge \omega_2,\qquad
		d\omega_2=\omega_1\wedge \omega_2,
		\]
		with $\fol=\ker(\omega_0)$ on the regular locus.
	\end{definition}
	
	A transverse projective structure yields a projective monodromy representation (defined up to conjugacy)
	\[
	\xi_{\mathrm{proj}}:\pi_1(U,q)\longrightarrow \mathrm{PGL}_2(\C),
	\qquad U=X\setminus D,
	\]
	obtained by parallel transport of the flat connection on $E$ and projectivisation.
	
	\subsection{Affine structures as reductions }\label{subsec:affine-reduction}
	
	A transverse affine structure gives a transverse projective structure whose monodromy
	lies in the affine subgroup
	\[
	\mathrm{Aff}(\C)=\C^*\ltimes \C\ \subset\ \mathrm{PGL}_2(\C),
	\qquad z\longmapsto a z + b.
	\]
	Geometrically, this is the case where the induced Riccati foliation on $\PP(E)$ admits an invariant section at infinity, so that the
	natural $\PP^1$-coordinate can be chosen on the affine chart $\C\subset \PP^1$, and the monodromy acts by affine transformations.
	In this situation, the \emph{linear part} of the affine monodromy is      the character $\rho$ of \Cref{subsec:affine-monodromy}, and the stack-theoretic base $[\C^*/\Gamma]$ records the multiplicative developing coordinate, that is, the linear part, rather than the full affine cocycle.
	
	\begin{center}
		\begin{tikzcd}[column sep=large,row sep=large]
			\pi_1(U,q)
			\arrow[r, "\xi_{\mathrm{proj}}"]
			\arrow[dr, bend right=18, "\xi"']
			&
			\mathrm{PGL}_2(\C)
			\\
			&
			\mathrm{Aff}(\C) \arrow[u, hook']
		\end{tikzcd}
	\end{center}
	
	The linear character $\rho$ is obtained from the affine representation and is studied separately through the rank-one logarithmic connection in \Cref{sec:log-RH}.

	\subsection{Gauge invariance}
	
	The linear character does not depend on the choice of frame. Indeed,
	under the change of trivialisation
	$\eta\mapsto\eta-d\log g$, with $g\in\mathcal O_U^*$, one has
	\[
	\exp\!\left(\int_\gamma(\eta-d\log g)\right)
	=\exp\!\left(\int_\gamma\eta\right)
	\]
	for every loop $\gamma$, since $\int_\gamma d\log g$ is an integral
	multiple of $2\pi i$. Thus $\rho$ and its image $\Gamma$ are
	independent of the chosen global trivialisation whenever one exists.
	
	The image of $\rho$ can be described using the meridians of $D$ and a set of additional generators of $H_1(U,\Z)$.
	
	\begin{proposition}\label[proposition]{prop:period-residue}
		Assume that $\eta$ is a global logarithmic connection form, that $D$ has finitely many irreducible components, and that $H_1(U,\Z)$ is finitely generated. For each irreducible component
		$D_i\subset D$, choose a smooth point of $D_i$ away from the other components of $D$,
		and let $\gamma_i$ be a positively oriented meridian around $D_i$, viewed as a based loop
		in $U$ after choosing a path from the base point $q$ to the corresponding transversal disk.
		Choose further loops $\delta_1,\dots,\delta_m$ such that the homology classes
		\[
		[\gamma_i],\ [\delta_j]\in H_1(U,\Z)
		\]
		generate $H_1(U,\Z)$ as an abelian group. Set
		$
		g_i:=\rho(\gamma_i)$,
		and $
		h_j:=\rho(\delta_j).
		$
		Then $\Gamma=\operatorname{im}(\rho)$ is generated by the elements $g_i$ and $h_j$. Equivalently,
		\[
		\Gamma=
		\left\langle
		\exp\!\bigl(2\pi i\,\Res_{D_i}(\eta)\bigr),\
		\exp\!\left(\int_{\delta_j}\eta\right)
		\right\rangle.
		\]
	\end{proposition}
	
	\begin{proof}
		Since the target $\C^*$ is abelian, the character $\rho$ factors through the abelianization
		of $\pi_1(U,q)$, hence through $H_1(U,\Z)$. Equivalently, because $d\eta=0$ on $U$, the quantity
		\[
		\exp\!\left(\int_\gamma\eta\right)
		\]
		depends only on the homology class $[\gamma]\in H_1(U,\Z)$.
		By assumption, the classes $[\gamma_i]$ and $[\delta_j]$ generate $H_1(U,\Z)$. Therefore,
		for every class $[\alpha]\in H_1(U,\Z)$ one may write
		\[
		[\alpha]=\sum_i a_i[\gamma_i]+\sum_j b_j[\delta_j]
		\qquad
		(a_i,b_j\in\Z).
		\]
		Since $\rho$ induces a homomorphism on $H_1(U,\Z)$, it follows that
		\[
		\rho([\alpha])
		=
		\prod_i \rho(\gamma_i)^{a_i}\prod_j \rho(\delta_j)^{b_j}
		=
		\prod_i g_i^{a_i}\prod_j h_j^{b_j}.
		\]
		Hence $\Gamma=\operatorname{im}(\rho)$ is generated by the elements $g_i$ and $h_j$.
		For a positively oriented meridian $\gamma_i$ around $D_i$, the residue formula for logarithmic
		$1$-forms gives
		\[
		\int_{\gamma_i}\eta=2\pi i\,\Res_{D_i}(\eta),
		\]
		and therefore
		$
		g_i=\rho(\gamma_i)=\exp\!\bigl(2\pi i\,\Res_{D_i}(\eta)\bigr).
		$
		Likewise,
		\[
		h_j=\rho(\delta_j)=\exp\!\left(\int_{\delta_j}\eta\right).
		\]
	\end{proof}
	
	\subsection{Endomorphisms of the linear monodromy group}
	
	Choose generators $(g_1,\ldots,g_k,h_1,\ldots,h_m)$ of $\Gamma$ as in \Cref{prop:period-residue}. They determine a surjective homomorphism
	\[
	\Phi:\Z^{k+m}\longrightarrow \Gamma,
	\qquad
	\Phi(n_1,\dots,n_k,t_1,\dots,t_m)
	=
	\prod_{i=1}^k g_i^{n_i}\prod_{j=1}^m h_j^{t_j},
	\]
	and hence an exact sequence of abelian groups
	\begin{equation}\label{eq:Gamma-presentation}
		0 \longrightarrow L \longrightarrow \Z^{k+m} \xrightarrow{\ \Phi\ } \Gamma \longrightarrow 0,
	\end{equation}
	where \(L=\ker(\Phi)\) is the subgroup of relations among the chosen generators.
	
	The presentation \eqref{eq:Gamma-presentation} describes the abstract endomorphisms of $\Gamma$. We later distinguish those induced by holomorphic semi-equivariant maps
	$\widetilde\psi(\gamma z)=\sigma(\gamma)\widetilde\psi(z)$.
	
	\begin{proposition}\label[proposition]{prop:EndGamma-matrix}
		There is a natural identification
		\[
		\End(\Gamma)\cong \End(\Z^{k+m},L)\big/\Hom(\Z^{k+m},L),
		\]
		where
		$
		\End(\Z^{k+m},L):=\{A\in \End(\Z^{k+m})\mid A(L)\subset L\}.
		$
		Equivalently, after choosing a basis of \(\Z^{k+m}\), one may describe \(\End(\Gamma)\) by integer matrices preserving the relation subgroup \(L\).
		Likewise, \(\Aut(\Gamma)\) is identified with the subgroup of classes represented by elements of \(\End(\Z^{k+m},L)\) whose induced endomorphism of \(\Z^{k+m}/L\) is invertible.
	\end{proposition}
	
	\begin{proof}
		Since \(\Gamma\cong \Z^{k+m}/L\) and \(\Z^{k+m}\) is a free abelian group, every endomorphism of \(\Gamma\) lifts to an endomorphism of \(\Z^{k+m}\). Any such lift must preserve \(L\), and therefore yields an element of \(\End(\Z^{k+m},L)\). In this way one obtains a surjective homomorphism
		\[
		\End(\Z^{k+m},L)\longrightarrow \End(\Gamma).
		\]
		Its kernel consists of those endomorphisms \(A\) of \(\Z^{k+m}\) whose induced map on \(\Z^{k+m}/L\) is zero. Equivalently, these are the endomorphisms satisfying
		$
		A(\Z^{k+m})\subset L.
		$
		Such maps form $\Hom(\Z^{k+m},L)$, giving the asserted quotient. A class is invertible precisely when its induced map on $\Z^{k+m}/L\cong\Gamma$ is invertible.
	\end{proof}
	
	The holomorphically realisable endomorphisms are classified in \Cref{sec:classification}.
	\section{Monodromy, quotient stacks, and affine development}\label{sec:stacky-base}
	
	We use quotient stacks and their action groupoids to distinguish the linear monodromy from the affine developing map. Morphisms of quotient stacks are understood through bibundles; strict semi-equivariant pairs concern a fixed presentation \cite{BehrendXu,Lerman,MoerdijkPronk,Olsson16}.
	
	We work with complex analytic stacks on the analytic \'{e}tale site of
	complex manifolds. Unless otherwise stated, the monodromy groups carry
	the discrete topology.
	
	\subsection{Quotient stacks and their presentations}
	
	\begin{definition} \label[definition]{def:quotient}
		Let a discrete group $\Gamma$ act holomorphically on a complex manifold $Y$ on the left.
		The quotient stack $[Y/\Gamma]$ has objects over $T$ given by a principal right $\Gamma$-bundle $P\to T$ and a holomorphic map $\phi:P\to Y$
		satisfying $\phi(p\cdot\gamma)=\gamma^{-1}\phi(p)$.
	\end{definition}
	
	\begin{definition}\label[definition]{def:action-groupoid}
		The action groupoid $Y\rtimes\Gamma\rightrightarrows Y$ has arrows $Y\times\Gamma$ with
		$
		s(y,\gamma)=y, $ and $ t(y,\gamma)=\gamma\cdot y.
		$
		Its classifying stack is $[Y/\Gamma]$.
	\end{definition}
	
	For quotient stacks by discrete groups, one may describe morphisms via bibundles (Morita morphisms). Strict pairs describe morphisms of the presenting action groupoids. General morphisms of quotient stacks are represented by bibundles and need not admit strict representatives on the chosen atlases. 
	
	\begin{definition} \label[definition]{def:bibundle}
		Let $\Gamma$ act on $Y$ and $\Gamma'$ on $Y'$.
		A $1$-morphism $[Y/\Gamma]\to[Y'/\Gamma']$ is represented by a manifold $P$ equipped with:
		\begin{enumerate}
			\item a principal right $\Gamma'$-bundle $\pi:P\to Y$;
			\item a left $\Gamma$-action on $P$ covering the given action on $Y$ and commuting with the right $\Gamma'$-action;
			\item a holomorphic map $\Phi:P\to Y'$ satisfying
			$\Phi(p\cdot\gamma')=\gamma'^{-1}\Phi(p)$ and $\Phi(\gamma\cdot p)=\Phi(p)$.
		\end{enumerate}
		A $2$-morphism is an isomorphism of these data. A strict pair $(f,\sigma)$ gives such a bibundle by taking $P=Y\times\Gamma'$ with $\gamma\cdot(y,g)=(\gamma y,\sigma(\gamma)g)$ and $\Phi(y,g)=g^{-1}f(y)$.
	\end{definition}

	\begin{definition} \label[definition]{def:strict}
		Let $\Gamma\subset \C^*$ be endowed with the discrete topology, independently of its subspace topology in $\C^*$, acting on $\C^*$ by multiplication, and consider $[\C^*/\Gamma]$.
		A \emph{strict endomorphism of the chosen action groupoid} is a pair $(\widetilde\psi,\sigma)$ consisting of a holomorphic map
		$\widetilde\psi:\C^*\to\C^*$ and a group homomorphism $\sigma:\Gamma\to\Gamma$ such that
		\[
		\widetilde\psi(\gamma z)=\sigma(\gamma)\,\widetilde\psi(z)\qquad \forall\,\gamma\in\Gamma,\ z\in\C^*.
		\]
		Two strict pairs $(\widetilde\psi,\sigma)$ and $(\widetilde\psi',\sigma)$ are \emph{$2$-isomorphic} if there exists $a\in\Gamma$ with
		$\widetilde\psi'(z)=a\,\widetilde\psi(z)$ for all $z$.
	\end{definition}
	
	\begin{proposition}\label[proposition]{prop:strictification}
		Let $\Gamma\subset\C^*$ be regarded as a discrete group. Strict endomorphisms of the action groupoid $\C^*\rtimes\Gamma$ are described by semi-equivariant pairs $(\widetilde\psi,\sigma)$ as in Definition~\ref{def:strict}. Two such pairs are naturally isomorphic if and only if their homomorphisms $\sigma$ agree and $\widetilde\psi'=a\widetilde\psi$ for some $a\in\Gamma$. This description does not assert that every Morita endomorphism of $[\C^*/\Gamma]$ admits a strict representative on the specified atlas.
	\end{proposition}
	
	\begin{proof}
		A natural transformation between the two functors is specified by a holomorphic map $a:\C^*\to\Gamma$ with $\widetilde\psi'=a\widetilde\psi$ and $a(\gamma z)\sigma(\gamma)=\sigma'(\gamma)a(z)$. Since $\C^*$ is connected and $\Gamma$ is discrete, $a$ is constant. The group is abelian, so the latter identity reduces to $\sigma=\sigma'$. A general Morita morphism is represented instead by a bibundle, whose pull-back to the atlas need not be trivial.
	\end{proof}
	
	Throughout the paper, whenever we write
	$
	\End([\C^*/\Gamma]),\ \Aut([\C^*/\Gamma]),\ \End([\PP^1/\Gamma]),
	$
	or analogous notation for quotient stacks, we mean $1$-endomorphisms in the analytic
	$2$-category of quotient stacks, considered up to $2$-isomorphism. The classifications below concerning semi-equivariant pairs apply to strict morphisms of the fixed presentation. Morita morphisms of the quotient stack need not possess such a representative; compare also \cite{BehrendXu,Noohi05}.
	\subsection{Developing maps and the multiplicative stack morphism}
	
	The following construction relates the connection to its developing maps.
	Let $(\omega,\eta)$ satisfy $d\omega=\omega\wedge\eta$ and $d\eta=0$ on $U$.
	On any simply connected open set $V\subset U$, choose a holomorphic primitive $\varphi$ of $\eta|_V$. With $m=e^{\varphi}$,
	\[
	d(m\omega)=m(\eta\wedge\omega+\omega\wedge\eta)=0.
	\]
	Consequently there is a holomorphic function $F_V$ such that $dF_V=m\omega$. On the regular locus, $\ker(dF_V)=\ker(\omega)=\fol|_V$.

	\begin{lemma}\label[lemma]{lem:dev}Let $p:\widetilde U\to U$ be the universal cover.
		There exists a holomorphic function $\Phi$ on $\widetilde U$ with $p^*\eta=d\Phi$.
		Setting $z:=e^\Phi\in \cO^*(\widetilde U)$, one has
		\[
		\frac{dz}{z}=p^*\eta,
		\qquad
		z(\gamma\cdot x)=\rho(\gamma)\,z(x)\quad(\gamma\in\pi_1(U)).
		\]
	\end{lemma}
	
	\begin{proof}
		Since $\widetilde U$ is simply connected and $d(p^*\eta)=p^*(d\eta)=0$, choose $\Phi$ with $d\Phi=p^*\eta$.
		For a deck transformation $\gamma$, the difference $\Phi\circ\gamma-\Phi$ is constant, and integrating $p^*\eta$ along a loop representing $\gamma$
		shows this constant equals $\int_\gamma \eta$. Exponentiating gives the formula for $z$.
	\end{proof}

	\begin{proposition}\label[proposition]{prop:canonicalF}
		Let $p:\widetilde U\to U$ be the universal covering and choose a primitive
		$\Phi$ of $p^*\eta$. Put $\Gamma=\operatorname{im}(\rho)$ and
		$P=\widetilde U/\ker(\rho)$. The function $z=e^\Phi$ determines a
		morphism of analytic stacks
		$
		F_z:U\longrightarrow[\C^*/\Gamma].
		$
		The underlying principal $\Gamma$-bundle $P\to U$ depends only on
		$\rho$. Replacing $\Phi$ by $\Phi+c$ replaces $F_z$ by its composite
		with multiplication by $e^c$ on the quotient stack; in general, the two
		morphisms need not be $2$-isomorphic.
	\end{proposition}
	
	\begin{proof}
		Use the right action $[x]\cdot\rho(\gamma)=[\gamma^{-1}x]$ on $P$.
		The equivariance relation of \Cref{lem:dev} gives
		\[
		\bar z([x]\cdot g)=g^{-1}\bar z([x]),\qquad g\in\Gamma.
		\]
		Thus $\bar z:P\to\C^*$ is equivariant in the convention of
		\Cref{def:quotient} and defines $F_z$. The primitive of $p^*\eta$
		is unique up to an additive constant; this changes $\bar z$ by the
		corresponding scalar. Since an isomorphism between the resulting
		$\Gamma$-equivariant maps on the connected cover must be multiplication
		by a constant element of $\Gamma$, the dependence claimed in the
		statement follows.
	\end{proof}
	
	The morphism $F_z$ describes the multiplicative coordinate, not a first
	integral of $\fol$: one has $dz/z=p^*\eta$, while the foliation is
	defined by $\omega$. Its affine developing map is obtained by integrating
	$e^\Phi p^*\omega$.
	
	\subsection{The linear classifying base, the translational class, and the affine torsor}
	
	The full affine monodromy contains a translational cocycle in addition to its linear character. Write $p:\widetilde U\to U$ for the universal cover.
	
	\begin{proposition}\label[proposition]{prop:linear-base-and-shadow}
		Let $\rho:\pi_1(U,q)\to\C^*$ be the linear part of the affine
		monodromy, and put $\Gamma=\operatorname{im}(\rho)$. The representation
		$\rho$ determines the following objects.
		\begin{enumerate}[label=\textup{(\roman*)},leftmargin=2.8em]
			\item A principal $\Gamma$-bundle $P_\rho\to U$ and its classifying
			morphism $b_\rho:U\to B\Gamma=[*/\Gamma]$.
			\item A rank-one complex local system $\mathbb L_\rho$ and the associated
			holomorphic line bundle with flat connection
			\[
			(\mathcal E_\rho,\nabla_\rho)
			=\bigl(\mathbb L_\rho\otimes_{\underline\C}\mathcal O_U,\nabla_\rho\bigr).
			\]
			The sheaf of horizontal sections of $\mathcal E_\rho$ is
			$\mathbb L_\rho$. The holomorphic bundle $\mathcal E_\rho$ is the
			pullback of $[\C/\Gamma]\to B\Gamma$ along $b_\rho$.
			\item A morphism $F:U\to[\C^*/\Gamma]$ lying over $b_\rho$ is
			equivalent to a nowhere-vanishing \emph{holomorphic} section of
			$\mathcal E_\rho$. Such a section need not be horizontal.
		\end{enumerate}
	\end{proposition}
	
	\begin{proof}
		The quotient $P_\rho=\widetilde U/\ker(\rho)$ is a principal
		$\Gamma$-bundle, since $\pi_1(U,q)/\ker(\rho)\simeq\Gamma$.
		The representation $\rho$ also determines a local system
		$\mathbb L_\rho$ with fibre $\C$. Extending its locally constant
		transition functions to holomorphic ones gives $\mathcal E_\rho$
		and its flat connection. Equivalently, the total space of
		$\mathcal E_\rho$ is the bundle associated to $P_\rho$ through the
		multiplicative action of $\Gamma$ on $\C$. This is the pullback of
		the universal bundle $[\C/\Gamma]\to B\Gamma$.
		A lift of $b_\rho$ to $[\C^*/\Gamma]$ is a $\Gamma$-equivariant
		holomorphic map $P_\rho\to\C^*$, or, equivalently, a nowhere-zero
		holomorphic section of $\mathcal E_\rho$. The section need not be
		parallel for $\nabla_\rho$.
	\end{proof}
	
	\begin{proposition}\label[proposition]{prop:translational-class-and-affine-torsor-bridge}
		Let $\xi:\pi_1(U,q)\to\Aff(\C)$ have linear part $\rho$ and write
		\[
		\xi(\gamma)(w)=\rho(\gamma)w+\tau(\gamma).
		\]
		Then $\tau$ is a cocycle with coefficients in $\mathbb L_\rho$ and
		determines a class $[\tau]\in H^1(U,\mathbb L_\rho)$. More precisely:
		\begin{enumerate}[label=\textup{(\roman*)},leftmargin=2.8em]
			\item The diagonal action of $\pi_1(U,q)$ on
			$\widetilde U\times\C$ through $\xi$ defines a holomorphic affine-line
			bundle $\mathcal A_\xi\to U$ with a flat affine connection.
			\item The sheaf $\mathscr T_\xi$ of local flat sections of
			$\mathcal A_\xi$ is a torsor under the additive local system
			$\mathbb L_\rho$. The corresponding holomorphic affine bundle is a
			torsor under the sheaf of holomorphic sections of $\mathcal E_\rho$.
			\item For the fixed linear monodromy $\rho$, flat affine torsors are
			classified by $H^1(U,\mathbb L_\rho)$; the class of
			$\mathscr T_\xi$ is $[\tau]$.
			\item The class $[\tau]$ vanishes if and only if a constant affine
			change of coordinate conjugates $\xi$ to its linear part.
		\end{enumerate}
	\end{proposition}
	
	\begin{proof}
		The identity $\xi(\gamma\delta)=\xi(\gamma)\xi(\delta)$ gives
		$
		\tau(\gamma\delta)
		=\tau(\gamma)+\rho(\gamma)\tau(\delta),
		$
		which is the cocycle condition. The quotient
		\[
		\mathcal A_\xi=(\widetilde U\times\C)/\pi_1(U,q),\qquad
		\gamma\cdot(x,w)=(\gamma x,\rho(\gamma)w+\tau(\gamma)),
		\]
		is a holomorphic affine bundle. Its locally constant transition
		functions preserve the standard affine connection on the fibres and
		therefore determine a flat affine connection. Differences of flat
		local sections transform by $\rho$, so the sheaf
		$\mathscr T_\xi$ is a torsor under $\mathbb L_\rho$; allowing
		holomorphic sections gives the analogous torsor under
		the sheaf of holomorphic sections of $\mathcal E_\rho$.
		
		Changing the origin of an affine fibre changes $\tau$ by a
		coboundary. Conversely, a coboundary provides a change of origin that
		removes the translation term. The usual classification of local-system
		torsors therefore gives $[\tau]\in H^1(U,\mathbb L_\rho)$ and the
		last assertion. This classification concerns the \emph{flat} torsor;
		forgetting the connection need not preserve its cohomology class.
	\end{proof}
	
	\begin{corollary}\label[corollary]{cor:affine-linear-bridge}
		Let $H=\operatorname{im}(\xi)\subset\Aff(\C)$ and
		$\Gamma=\operatorname{im}(\rho)\subset\C^*$. Suppose that a
		global presentation $(\omega,\eta)$ and developing maps
		$z=e^\Phi$ and $f$ have been chosen on $\widetilde U$, with
		\[
		d\Phi=p^*\eta,\qquad df=e^\Phi p^*\omega,
		\qquad f(\gamma x)=\xi(\gamma)(f(x)).
		\]
		Then $f$ determines a holomorphic section $s_\xi$ of the flat
		affine bundle $\mathcal A_\xi$. The corresponding map
		$\mathrm{Dev}_\xi:U\to[\C/H]$ and the multiplicative map
		$F_z:U\to[\C^*/\Gamma]$ lie over the same classifying map
		$b_\rho:U\to B\Gamma$.
		\[
		\begin{tikzcd}[column sep=large,row sep=large]
			& {[\C/H]} \arrow[d] \\
			U \arrow[ur,"\mathrm{Dev}_\xi"]
			\arrow[r,"b_\rho"'] \arrow[dr,"F_z"'] & B\Gamma \\
			& {[\C^*/\Gamma]} \arrow[u]
		\end{tikzcd}
		\]
		The affine map records both the flat affine torsor
		$\mathscr T_\xi$ and its holomorphic developing section
		$s_\xi$; the multiplicative map records the holomorphic section
		of $\mathcal E_\rho$ determined by $z$. Neither section is
		required to be horizontal.
	\end{corollary}
	
	\begin{proof}
		The $\xi$-equivariance of $f$ makes
		$x\mapsto[x,f(x)]$ a well-defined holomorphic section of
		$\mathcal A_\xi$. Equivalently, $f$ descends through the principal
		$H$-bundle $\widetilde U/\ker(\xi)$ to a morphism
		$U\to[\C/H]$. Projection of the affine representation
		$H\to\Gamma$ gives a map $[\C/H]\to B\Gamma$.
		Likewise, the equivariance of $z$ gives $F_z$ and its projection to
		$b_\rho$ by \Cref{prop:canonicalF,prop:linear-base-and-shadow}.
		These projections agree up to the canonical $2$-isomorphisms induced
		by the underlying principal bundles. There is generally no map
		$[\C/H]\to[\C^*/\Gamma]$ compatible with these projections:
		translations in $H$ do not preserve a multiplicative coordinate.
	\end{proof}

	\begin{remark}
		Let $M\subset U$ be the regular locus of $\mathcal F$, and let
		$\mathcal G\rightrightarrows M$ be an analytic groupoid whose
		orbits are the leaves, assuming that $[M/\mathcal G]$ exists
		in the chosen analytic category.
		By the defining descent property of the quotient stack,
		a morphism $F:M\to\mathcal Y$ descends to $[M/\mathcal G]$
		precisely when it admits a $2$-isomorphism
		\[
		\alpha:s^*F\xRightarrow{\sim}t^*F
		\]
		satisfying the unit and cocycle identities.
		Once this descent datum is fixed, the factorisation is unique
		up to a unique compatible $2$-isomorphism.
		In particular, a holomorphic map to a complex manifold
		that is constant on the leaves descends uniquely.
		For a stack-valued map, leafwise constancy alone does not
		provide the required compatible $2$-isomorphisms.
		The multiplicative map \(F_z\), however need not descend
		to $[M/\mathcal G]$, since it is not generally constant
		along the leaves.
	\end{remark}
	
	\subsection{Quotient-stack morphisms and Kummer base change}
	We compare maps of the multiplicative quotient with Kummer base change along the boundary. The former concern the linear character; the latter change the meridional residues \cite{Cadman,CSSST17,VistoliTalpo18}. 
	We adopt throughout the Morita--bibundle formalism of differentiable and complex stacks, following \cite{BehrendXu,Lerman,MoerdijkPronk,Olsson16}. Let $\Gamma$ be a discrete group acting holomorphically on $\PP^1$ by transformations of the form $z \mapsto \gamma z$, and suppose that this action preserves the distinguished subset $\{0,\infty\}$.
	We always regard $\Gamma$ with the discrete topology (so the quotient stack is defined even when the action is not proper).
	
	\begin{definition} \label[definition]{def:morita-mor}
		A \emph{Morita morphism} $[\PP^1/\Gamma]\to [\PP^1/\Gamma']$ is a $1$-morphism of stacks represented by a bibundle
		in the sense of \Cref{def:bibundle}; $2$-morphisms are bibundle isomorphisms.
	\end{definition}

	Strict semi-equivariant representatives permit explicit functional-equation arguments. A morphism represented only by a bibundle requires a separate descent argument; compare \cite{BehrendXu,Noohi05}.
	
	Finite ramification along the boundary is described by Kummer and root constructions \cite{Cadman}. On $\PP^1$ these are locally given by adjoining $N$-th roots at $\{0,\infty\}$.
	
	\begin{definition} \label[definition]{def:kummer-ref}
		Fix $N\ge 1$. Consider the \emph{$N$-th power-map presentation} associated with $[\PP^1/\Gamma]$:
		$$
		[\PP^1/(\mu_N\times\Gamma)],
		$$
		where $(\zeta,\gamma)\in\mu_N\times\Gamma$ acts by $(\zeta,\gamma)\cdot z = (\zeta\gamma)\,z$.
		Its restriction to the open atlas is $[\C^*/(\mu_N\times\Gamma)]$.
	\end{definition}
	
	\begin{proposition}\label[proposition]{prop:kummer-map}
		The power map $z\mapsto z^N$ and the homomorphism $(\zeta,\gamma)\mapsto\gamma^N$ induce a morphism
		\[
		\kappa_N:[\PP^1/(\mu_N\times\Gamma)]\longrightarrow[\PP^1/\Gamma].
		\]
		For $N>1$ this morphism is not representable. Its restriction to the open locus is a morphism of quotient stacks. Separately, the local logarithmic base change $t=w^N$ multiplies a residue by $N$, and therefore raises the corresponding meridional multiplier to its $N$-th power.
	\end{proposition}
	
	\begin{proof}
		Equivariance follows from $((\zeta,\gamma)z)^N=\gamma^N z^N$. At $0$ and $\infty$ the homomorphism on stabilisers sends $(\zeta,\gamma)$ to $\gamma^N$, and its kernel contains $\mu_N$. It is therefore not injective when $N>1$, which precludes representability. For the separate logarithmic pullback, $t=w^N$ gives $dt/t=N\,dw/w$, proving the residue and multiplier statement.
	\end{proof}

	For $\Gamma=\{1\}$, the quotient $[\PP^1/\mu_N]$ is the $N$-th root stack along $\{0\}+\{\infty\}$.
	The quotient $[\PP^1/(\mu_N\times\Gamma)]$ is used here as a presentation of the power-map construction. For general $\Gamma$ it need not be the ordinary root stack of $[\PP^1/\Gamma]$: if $\mu_N\cap\Gamma\ne\{1\}$, its action has a nontrivial kernel and hence generic inertia. Root constructions and their relation to Kato--Nakayama spaces are discussed in \cite{Cadman,CSSST17,VistoliTalpo18}.

	We now describe the affine developing map associated
	with the transversely affine structure.
	Let $\omega$ and $\eta$ be holomorphic one-forms on a
	connected complex manifold $U$, with $\omega\not\equiv0$,
	$d\eta=0$, and $d\omega=\omega\wedge\eta$.
	Write $\mathcal F$ for the foliation defined by $\omega$,
	let $p:\widetilde U\to U$ be the universal cover, and choose
	$\Phi$ with $d\Phi=p^*\eta$.
	
	\begin{theorem}\label{thm:geo-singer-ops}
		There exists a holomorphic function
		$f:\widetilde U\to\mathbb C$ satisfying
		\[
		df=e^\Phi p^*\omega,
		\qquad
		f\circ\gamma=\rho(\gamma)f+\tau(\gamma)
		\quad\text{for every }\gamma\in\pi_1(U),
		\]
		where $\rho(\gamma)=\exp\!\left(\int_\gamma\eta\right)$
		and $\tau(\gamma)\in\mathbb C$.
		On the locus where $p^*\omega$ does not vanish,
		the fibres of $f$ locally coincide with the leaves
		of $p^*\mathcal F$.
	\end{theorem}
	
	\begin{proof}
		Since $d(e^\Phi p^*\omega)=e^\Phi(p^*\eta\wedge p^*\omega+p^*\omega\wedge p^*\eta)=0$, the form is exact on the simply connected cover, giving $f$. At a regular point the nonzero scalar $e^\Phi$ does not change the kernel of $p^*\omega$. For a deck transformation $\gamma$, the difference $\Phi\circ\gamma-\Phi$ is the constant $\int_\gamma\eta$. Consequently $d(f\circ\gamma)=\rho(\gamma)\,df$, so $f\circ\gamma=\rho(\gamma)f+\tau(\gamma)$ for a constant $\tau(\gamma)$. 
	\end{proof}
	
	Local submersive first integrals for the same foliation
	differ by a holomorphic change of one variable.
	Indeed, choosing coordinates with $f=t$, the equality
	$\ker(df)=\ker(dg)$ gives $g=h(t)$, with $h'\neq0$
	because $g$ is a submersion.
	
	The multiplicative coordinate $e^\Phi$ describes the linear character of the transverse affine structure. The affine primitive $f$ determines its regular leaves. Kummer changes of boundary parameter alter the local meridional multiplier as above; they do not, by themselves, provide a global classification of first integrals or of Morita endomorphisms.
	\section{Strict maps of the multiplicative quotient}\label{sec:classification}
	
	We classify strict semi-equivariant maps of the multiplicative presentation, distinguishing finite, properly discontinuous and non-discrete monodromy groups. The elliptic quotient requires a separate description of its intrinsic endomorphisms.
	
	\subsection{Subgroups of the multiplicative group}
	
	In this subsection we study finitely generated subgroups $\Gamma\subset\C^*$ and the topology of their quotients.
	
	\begin{definition} \label[definition]{def:logarg}
		Let
		$
		\ell:\C^*\longrightarrow \R\times(\R/2\pi\Z),
		$
		and 
		$\ell(z)=(\log|z|,\arg z).
		$
		Define $A_\Gamma=\operatorname{pr}_1(\ell(\Gamma))\subset\R$ and $B_\Gamma=\operatorname{pr}_2(\ell(\Gamma))\subset\R/2\pi\Z$. The subgroup $\ell(\Gamma)$ need not be the product of these projections.
	\end{definition}
	
	\begin{proposition}\label[proposition]{prop:discrete-subgroups}
		Let $\Gamma\subset \C^*$ be a finitely generated subgroup. Then the following hold.
		
		\begin{enumerate}[label=\textup{(\roman*)},leftmargin=2.8em]
			\item If $A_\Gamma=\{0\}$, then $\Gamma\subset \Sone$, and $\overline{\Gamma}$ is finite if and only if $B_\Gamma$ is finite; otherwise $\overline{\Gamma}=\Sone$.
			
			\item If $A_\Gamma\neq \{0\}$, then $\Gamma$ contains an element $q$ with $|q|\neq 1$. In particular, every non-fixed orbit of the action
			$\Gamma\curvearrowright \PP^1$
			accumulates at $0$ and $\infty$.
			
			\item The following are equivalent:
			\begin{enumerate}[label=\textup{(\alph*)},leftmargin=2.4em]
				\item the action $\Gamma\curvearrowright \C^*$ is properly discontinuous;
				\item $\Gamma$ is discrete in the Euclidean topology of $\C^*$;
				\item either $\Gamma=\mu_N$ for some $N\ge 1$, or
				$
				\Gamma=\mu_N\cdot q^\Z$
				for some $N\ge 1$   and  $q\in\C^*$   with $|q|\neq 1$.
			\end{enumerate}
			
			\item If $\Gamma$ is torsion-free and discrete, then either $\Gamma=\{1\}$ or $\Gamma=q^\Z$ for some $q\in\C^*$ with $|q|\neq 1$.
		\end{enumerate}
	\end{proposition}
	
	\begin{proof}
		Statement \textup{(i)} is the standard classification of finitely generated subgroups of the compact group $\Sone$: such a subgroup is either finite or dense in $\Sone$.
		
		\noindent For \textup{(ii)}, if $A_\Gamma\neq \{0\}$, then by definition the radial part of $\Gamma$ is nontrivial, so there exists $q\in\Gamma$ with $|q|\neq 1$. Iterating $q$ or $q^{-1}$ shows that the orbit of any point of $\C^*$ accumulates at $0$ and at $\infty$ in $\PP^1$.
		
		\noindent \textup{(a) \(\Rightarrow\) (b).}
		If $\Gamma$ were not discrete in $\C^*$, there would exist a sequence of distinct elements $\gamma_n\in\Gamma$ converging to some $\gamma\in\Gamma$. Then $\gamma^{-1}\gamma_n\to 1$, so for any sufficiently small compact neighbourhood $K$ of $1$ one would have
		$
		(\gamma^{-1}\gamma_n)\cdot K\cap K\neq\varnothing
		$
		for infinitely many $n$, contradicting proper discontinuity.
		
		\noindent \textup{(b) $\Rightarrow$ (c).}
		Under the universal covering $\C\to\C^*$, $w\mapsto e^w$, the full inverse image of $\Gamma$ is a discrete subgroup $L\subset\C$ containing $2\pi i\Z$. A discrete subgroup of the real plane has rank at most two. Its intersection with $i\R$ contains $2\pi i\Z$. Since $L$ has rank at most two, this intersection has rank one, and the quotient of $L$ by this intersection has rank at most one. Consequently the image of $L$ under the real-part homomorphism is zero or a cyclic subgroup $\lambda\Z$, with $\lambda>0$. The kernel of the modulus map on $\Gamma$ is a finite subgroup of $S^1$, hence equals $\mu_N$ for some $N$. If the radial image vanishes, $\Gamma=\mu_N$. Otherwise choose $q\in\Gamma$ with $\log|q|=\lambda$. Every $\gamma\in\Gamma$ then has a unique expression $\gamma=\zeta q^m$, with $\zeta\in\mu_N$ and $m\in\Z$, proving the claim.
		
		\noindent\textup{(c) \(\Rightarrow\) (a).}
		If $\Gamma=\mu_N$, the action is finite and therefore properly discontinuous. If
		\[
		\Gamma=\mu_N\cdot q^\Z
		\ \text{with }|q|\neq 1,
		\]
		let $K\subset \C^*$ be compact. Since multiplication by $q^m$ rescales the modulus exponentially, only finitely many integers $m$ satisfy
		$
		q^mK\cap K\neq \varnothing.
		$
		As the torsion factor $\mu_N$ is finite, it follows that only finitely many elements of $\Gamma$ move $K$ so that it still meets itself. Hence the action is properly discontinuous.
		
		\noindent Finally, \textup{(iv)} follows from \textup{(iii)}: a torsion-free discrete subgroup is either trivial or generated by an element of modulus different from $1$.
	\end{proof}

	Consider $D_i$ the irreducible components of the divisor $D$ and $\eta$ the meromorphic 1-form of the transversely affine structure on the foliation. Write each meridian multiplier as
	\[
	\lambda_i:=\exp\big(2\pi i\,\Res_{D_i}(\eta)\big)=\exp(2\pi i(a_i+i b_i))=e^{-2\pi b_i}\cdot e^{2\pi i a_i},
	\]
	so $|\lambda_i|=e^{-2\pi b_i}$ and $\lambda_i/|\lambda_i|=e^{2\pi i a_i}$.
	Define the finitely generated additive subgroup
	\[
	B:=\langle b_1,\dots,b_k\rangle_{\Z}\subset\R,
	\qquad
	M_D:=\langle|\lambda_1|,\dots,|\lambda_k|\rangle\subset\R_{>0}.
	\]
	Thus $M_D=\exp(-2\pi B)$ is a subgroup of $|\Gamma|:=\{|\gamma|:\gamma\in\Gamma\}$. Period multipliers need not belong to $M_D$.
	
	\begin{proposition}\label[proposition]{prop:rank2-dense-nonT0}
		Assume that $\rk(B)\ge 2$. Then $B$ and $M_D=\exp(-2\pi B)$ are dense in $\R$ and $\R_{>0}$, respectively. Consequently $|\Gamma|$ is dense in $\R_{>0}$. In particular, the orbit space $\C^*/\Gamma$ is not $T_0$, hence in particular not Hausdorff.
	\end{proposition}
	
	\begin{proof}
		The additive subgroup $B$ contains two elements whose ratio is irrational, so the usual Diophantine approximation argument shows that $B$ is dense in $\R$. The exponential map $t\mapsto e^{-2\pi t}$ gives density of $M_D$ in $\R_{>0}$. Since $M_D\subseteq|\Gamma|$, the latter is also dense. In particular, $\Gamma$ cannot be closed in $\C^*$: otherwise the closed subgroup theorem and finite generation would make $\Gamma$ discrete, contradicting the density of its radial moduli. The quotient $\C^*/\Gamma$ is not $T_0$, since the coset of the identity has closure containing distinct cosets when $\Gamma$ is not closed.
	\end{proof}
	The action groupoid still defines an analytic quotient stack when the
	orbit space is not Hausdorff. For finite monodromy
	$\Gamma=\mu_N$, the compactified quotient is a Deligne--Mumford
	curve. If $\Gamma$ is infinite and discrete, it has the form
	$\mu_N\cdot q^\Z$, with $|q|\ne1$, and $\C^*/\Gamma$ is an elliptic
	curve. On the other hand, an infinite subgroup of $S^1$ is dense in
	the circle, while the hypothesis $\rk(B)\ge2$ produces dense radial
	multipliers. These are useful instances of non-closed monodromy
	subgroups; they do not exhaust the ways in which non-discreteness can
	arise from the remaining period multipliers.
	
	Even when $H_1(U,\Z)$ is finitely generated, so that $\Gamma$ is a
	finitely generated abelian group, its embedding in $\C^*$ need not be
	discrete.
	
	The multiplication action of $\Gamma$ on $\C^*$ is a translation action.
	For such actions the orbit space is Hausdorff precisely when the subgroup is closed.
	
	\begin{lemma}\label[lemma]{lem:hausdorff-coset}
		Let $G$ be a Hausdorff topological group and $H\subset G$ a subgroup.
		Then the quotient space $G/H$ is Hausdorff if and only if $H$ is closed in $G$.
	\end{lemma}
	
	Applied to the case \(G=\C^*\), this shows that a monodromy subgroup which is not closed necessarily gives rise to a non-Hausdorff orbit space. For non-closed $\Gamma$, the orbit space is not a Hausdorff Riemann surface. Since \(\Gamma\) is finitely generated, the question of closedness may moreover be decided explicitly.
	
	\begin{proposition}\label[proposition]{prop:Gamma-closed}
		Let $\Gamma\subset\C^*$ be a finitely generated subgroup.
		Then the following are equivalent:
		\begin{enumerate}[label=(\alph*),leftmargin=2.2em]
			\item $\Gamma$ is closed in $\C^*$.
			\item $\Gamma$ is discrete in $\C^*$.
			\item Either $\Gamma$ is finite (hence $\Gamma=\mu_N$), or $\Gamma$ is virtually cyclic of the form
			$
			\Gamma = \mu_N\cdot q^\Z
			$
			for some $N\ge 1$ and some $q\in\C^*$ with $|q|\neq 1$.
		\end{enumerate}
		Consequently, $\C^*/\Gamma$ is Hausdorff if and only if one is in case (c).
	\end{proposition}
	
	\begin{proof}
		A closed subgroup of the Lie group $\C^*$ is a Lie subgroup. Since $\Gamma$ is finitely generated, it is countable and cannot contain a positive-dimensional Lie subgroup. Hence a closed $\Gamma$ is discrete. Conversely every discrete subgroup of $\C^*$ is closed: otherwise a convergent sequence of distinct elements would, after taking quotients, give nonidentity elements tending to $1$, contrary to discreteness. The classification of discrete subgroups is Proposition~\ref{prop:discrete-subgroups}(iii). Hausdorffness of the quotient is Lemma~\ref{lem:hausdorff-coset}.
	\end{proof}

	If $\Gamma=q^\Z$ with $|q|=1$ and $q$ not a root of unity, then $\Gamma$ is not discrete (its powers are dense in $\Sone$),
	hence not closed, and $\C^*/\Gamma$ fails to be Hausdorff by \Cref{lem:hausdorff-coset}.
	The analytic quotient stack $[\C^*/\Gamma]$ remains defined even when the orbit space is non-Hausdorff.

	\begin{corollary}\label[corollary]{cor:curve-then-algebraic}
		Let $\Gamma\subset\C^*$ be finitely generated and closed. The quotient $\C^*/\Gamma$ is isomorphic to $\C^*$ when $\Gamma$ is finite and to an elliptic curve when $\Gamma$ is infinite. In particular, these quotients are algebraic curves.
	\end{corollary}
	
	\begin{proof}
		If $\Gamma=\mu_N$, the map $z\mapsto z^N$ identifies $\C^*/\Gamma$ with $\C^*$. Otherwise Proposition~\ref{prop:discrete-subgroups} gives $\Gamma=\mu_N\cdot q^\Z$ with $|q|\ne1$. The same power map induces
		\[
		\C^*/(\mu_N\cdot q^\Z)\simeq\C^*/(q^N)^\Z
		\simeq\C/(2\pi i\Z+N\log(q)\Z),
		\]
		which is an elliptic curve. The final assertion follows from the algebraicity of elliptic curves.
	\end{proof}

	\subsection{Normal form}
	
	\begin{proposition}\label[proposition]{prop:hol-normal-aut-Cstar}
		Let $f:\C^*\to\C^*$ be holomorphic. Then there exist $n\in\Z$ and $H\in\cO(\C^*)$ such that
		$
		f(z)=z^n\exp(H(z)).
		$
		The integer $n$ is unique and $H$ is unique up to addition of a constant in $2\pi i\Z$. Moreover, $f$ is a biholomorphism if and only if
		$
		f(z)=az
		$
		or
		$
		f(z)=a/z
		$
		for some $a\in\C^*$.
	\end{proposition}
	
	\begin{proof}
		Let $\pi:\C\to\C^*$, $\pi(w)=e^w$, be the universal cover. Choose a holomorphic function $h:\C\to\C$ such that
		$
		\exp(h(w))=f(e^w).
		$
		Then
		$
		h(w+2\pi i)-h(w)\in 2\pi i\Z.
		$
		By connectedness, this difference is constant, say $2\pi i n$.
		Thus $h(w)-nw$ is $2\pi i$-periodic and descends to a holomorphic function $H\in\cO(\C^*)$, giving
		$
		f(z)=z^n\exp(H(z)).
		$
		The integer $n$ is the winding number of $f$ about the origin. If $H_1,H_2$ yield the same factorisation, then $e^{H_1-H_2}=1$, so $H_1-H_2$ is a constant in $2\pi i\Z$.
		
		Assume now that $f$ is a biholomorphism. Then $f^{-1}$ is holomorphic on $\C^*$, so $f$ has neither zeros nor critical points on $\C^*$.
		Applying the same normal form to $f^{-1}$, one sees that the degree of $f$ on $\pi_1(\C^*)\cong\Z$ must be $\pm 1$.
		Hence $n=\pm1$. Since $f$ is injective, neither $0$ nor $\infty$ can be an essential singularity: this follows from the great Picard theorem. Thus $f$ extends to a Möbius transformation preserving $\{0,\infty\}$, and therefore $f(z)=az$ or $f(z)=a/z$ for some $a\in\C^*$.
		Conversely, both maps are biholomorphisms.
	\end{proof}
	
	Combining \Cref{prop:strictification,prop:hol-normal-aut-Cstar} gives the strict automorphisms of the chosen presentation:
	they are induced by $\widetilde\psi(z)=az$ (with $\sigma=\mathrm{id}$) and $\widetilde\psi(z)=a/z$ (with $\sigma(\gamma)=\gamma^{-1}$),
	modulo the $2$-isomorphism ambiguity $a\sim \gamma a$.
	Thus their group of natural-isomorphism classes is $(\C^*/\Gamma)\rtimes(\Z/2)$, where $\Z/2$ acts by inversion.
	We make this precise in \Cref{sec:AutBGamma}.
	
	\begin{lemma}\label[lemma]{lem:entire-periodicity-rigidity}
		Let $h:\C\to\C$ be entire.
		\begin{enumerate}[label=\textup{(\roman*)},leftmargin=2.6em]
			\item If there exist $\omega_1,\omega_2\in\C$ which are $\R$-linearly independent and satisfy
			$
			h(w+\omega_1)=h(w)$ and $
			h(w+\omega_2)=h(w)
			$
			for all $w$, then $h$ is constant.
			
			\item If there exists a sequence $u_n\to 0$ with $u_n\neq 0$ such that
			$
			h(w+u_n)-h(w)
			$
			is constant in $w$ for every $n$, then $h$ is affine-linear.
			If moreover $h$ is $2\pi i$-periodic, then $h$ is constant.
		\end{enumerate}
	\end{lemma}
	
	\begin{proof}
		For \textup{(i)}, the two periods generate a lattice $\Lambda=\Z\omega_1+\Z\omega_2$. The function $h$ is bounded on a compact fundamental parallelogram and hence on $\C$. Liouville's theorem gives the result.
		
		For \textup{(ii)}, fix $n$ and write
		$
		h(w+u_n)=h(w)+c_n
		$
		with $c_n\in\C$.
		Dividing by $u_n$ and letting $u_n\to 0$, one obtains by uniform convergence on compact sets that $h'(w)$ is constant.
		Hence $h$ is affine-linear.
		If $h$ is $2\pi i$-periodic, then its linear term must vanish, so $h$ is constant.
	\end{proof}
	
	\subsection{Holomorphic rigidity for infinite monodromy}
	
	An infinite subgroup $\Gamma\subset\C^*$ either contains an element
	of modulus different from $1$, or is an infinite subgroup of $S^1$.
	In the latter case there are nonidentity elements of $\Gamma$
	arbitrarily close to $1$. We use these two alternatives to establish
	rigidity of semi-equivariant holomorphic maps.
	
	\begin{proposition}\label[proposition]{proposition:hol-rigidity}
		Assume that $\Gamma\subset\C^*$ is infinite.
		Let $(\widetilde\psi,\sigma)$ be a strict holomorphic endomorphism of $[\C^*/\Gamma]$, i.e.\ $\widetilde\psi:\C^*\to\C^*$ holomorphic and
		$\sigma:\Gamma\to\Gamma$ a homomorphism with $\widetilde\psi(\gamma z)=\sigma(\gamma)\widetilde\psi(z)$.
		Then there exist $a\in\C^*$ and $n\in\Z$ such that
		$
		\widetilde\psi(z)=a z^n, $ and $ \sigma(\gamma)=\gamma^n$ for all $\gamma\in\Gamma.
		$
	\end{proposition}
	
	\begin{proof}
		Let $(\widetilde\psi,\sigma)$ be semi-equivariant and write $\widetilde\psi(z)=z^n\exp(H(z))$ by \Cref{prop:hol-normal-aut-Cstar}.
		Lift $H$ to an entire function $h:\C\to\C$ via $h(w)=H(e^w)$, so $h$ is $2\pi i$-periodic.
		Semi-equivariance gives $\exp(H(\gamma z)-H(z))=\sigma(\gamma)\gamma^{-n}$, hence for each fixed $\gamma\in\Gamma$ the difference
		$H(\gamma z)-H(z)$ is constant in $z$. Writing $\gamma=e^{u}$ (choosing any logarithm $u$), this becomes
		\[
		h(w+u)-h(w)=c(u)\quad\text{(constant in $w$)}.
		\]
		
		If $\Gamma$ contains $q$ with $|q|\ne1$, choose a logarithm $u_0=\log q$. Then $h(w+u_0)-h(w)=c_0$ for some constant $c_0$.
		Differentiation gives $h'(w+u_0)=h'(w)$. Thus $h'$ has the two $\R$-linearly independent periods $u_0$ and $2\pi i$, because $|q|\ne1$. By Liouville's theorem, $h'$ is constant. The $2\pi i$-periodicity of $h$ forces this constant to vanish, so $h$ and $H$ are constant.
		
		Otherwise $\Gamma$ is an infinite subgroup of $S^1$, so there is a sequence $\gamma_m\to1$ with $\gamma_m\ne1$.
		Choose logarithms $u_m=\log\gamma_m$ with $u_m\to 0$.
		Then $h(w+u_m)-h(w)$ is constant in $w$ for each $m$.
		By \Cref{lem:entire-periodicity-rigidity}, $h$ is constant (because it is $2\pi i$-periodic). Hence $H$ is constant.
		
		In either case, $\widetilde\psi(z)=a z^n$ with $a=e^{H}\in\C^*$.
		Substituting into semi-equivariance yields $a(\gamma z)^n=\sigma(\gamma) a z^n$, hence $\sigma(\gamma)=\gamma^n$ for all $\gamma\in\Gamma$.
	\end{proof}

	\begin{corollary}\label[corollary]{cor:realizability-wild}
		Suppose that $\Gamma\subset\C^*$ is infinite. The endomorphisms of $\Gamma$ induced by strict holomorphic self-maps of $\C^*$ are precisely the power maps
		\[
		\sigma_n:\Gamma\longrightarrow\Gamma,\qquad \gamma\longmapsto\gamma^n,
		\]
		where $n\in\Z$. Such a power map is an automorphism exactly when it is bijective on $\Gamma$.
		If, in addition, $\Gamma$ is finitely generated, the only power maps which are automorphisms are the identity and inversion.
	\end{corollary}
	\begin{proof}
		The classification of the induced endomorphisms follows from \Cref{proposition:hol-rigidity}, and every power map occurs by taking $\widetilde\psi(z)=z^n$. The bijectivity condition is immediate. If $\Gamma$ is finitely generated, it is a finite direct sum of cyclic groups. Since $\Gamma$ is infinite, its free rank is positive. A power map acts as multiplication by $n$ on this free summand; invertibility forces $n=1$ or $n=-1$, and both choices are automorphisms of $\Gamma$.
	\end{proof}

	For the compactified presentation the same rigidity gives the following result.
	
	\begin{proposition}\label[proposition]{proposition:mero-rigidity}
		Assume that $\Gamma\subset\C^*$ is infinite.
		If $R\in\C(z)$ is nonzero and $\sigma:\Gamma\to\Gamma$ satisfy the semi-equivariance relation
		$
		R(\gamma z)=\sigma(\gamma)R(z)\ (\gamma\in\Gamma,\ z\in\C^*),
		$
		then $R(z)=c z^n$ for some $c\in\C^*$ and $n\in\Z$.
	\end{proposition}
	
	\begin{proof}
		Let $\Div(R)$ be the divisor of zeros and poles of $R$ on $\PP^1$, with finite support.
		Semi-equivariance implies $\Div(R)$ is $\Gamma$-invariant.
		Since $\Gamma$ is infinite, the only finite $\Gamma$-invariant subset of $\PP^1$ is contained in $\{0,\infty\}$
		(because any point of $\C^*$ has infinite orbit accumulating in $\PP^1$).
		Hence $\Div(R)\subset\{0,\infty\}$ and $R$ has no zeros or poles in $\C^*$.
		Therefore $R$ is a monomial $c z^n$.
	\end{proof}

	\subsection{Finite monodromy case}
	
	Suppose that $\Gamma=\mu_N$ is the group of $N$-th roots of unity. The quotient $[\PP^1/\mu_N]$ is a Deligne--Mumford curve, canonically isomorphic to the
	$N$-th root stack of $(\PP^1,\{0\}+\{\infty\})$ \cite{BehrendNoohi06,Cadman}.
	In this regime the $N$-th power of the multiplicative coordinate is single-valued. It records the linear connection, but it need not be a first integral of $\fol$.
	
	\begin{proposition}\label[proposition]{prop:Nth-power}
		Let $D\subset X$ be a simple normal crossings divisor, and assume that $\eta$ is a closed logarithmic $1$-form on $X$. Suppose that the image of its period character $\rho$ is contained in $\mu_N$. Then the $N$-th power of the local multiplicative developing coordinate descends to a single-valued holomorphic function $f:U\to\C^*$. It extends meromorphically across $D$.
	\end{proposition}
	
	\begin{proof}
		On the universal cover, $z(\gamma x)=\rho(\gamma)z(x)$, hence $z^N$ is invariant under deck transformations. Locally near $D$, write $\eta=\sum_i\alpha_i\,dt_i/t_i+dh$ with $h$ holomorphic. Since every meridian has multiplier in $\mu_N$, one has $N\alpha_i\in\Z$. Thus $z^N$ is, up to a nonzero constant, $\exp(Nh)\prod_i t_i^{N\alpha_i}$ and extends meromorphically. The last assertion follows because $dz/z=\eta$ whereas the foliation is defined by $\omega$.
	\end{proof}
	
	\begin{proposition}\label[proposition]{proposition:End-muN}
		Strict holomorphic endomorphisms of $[\C^*/\mu_N]$ are pairs $(\widetilde\psi,\sigma)$ with
		\[
		\sigma(\zeta)=\zeta^m,\qquad
		\widetilde\psi(z)=z^m g(z^N),
		\]
		where $m$ is the unique representative in $\{0,\ldots,N-1\}$ of the exponent defining $\sigma$, and $g:\C^*\to\C^*$ is holomorphic. For this fixed representative, two such pairs are $2$-isomorphic if and only if their functions $g$ differ by multiplication by an element of $\mu_N$.
	\end{proposition}
	
	\begin{proof}
		Any homomorphism $\sigma:\mu_N\to\mu_N$ is of the form
		$
		\sigma(\zeta)=\zeta^m
		$
		for a unique $m\in\{0,\ldots,N-1\}$. Thus semi-equivariance for a strict endomorphism $(\widetilde\psi,\sigma)$ is the condition
		\[
		\widetilde\psi(\zeta z)=\zeta^m\widetilde\psi(z),
		\]
		where $\zeta\in\mu_N$. Equivalently,
		$
		\widetilde\psi(z)/z^m
		$
		is $\mu_N$-invariant. Since the quotient map
		$
		\C^*\longrightarrow \C^*,\  z\longmapsto z^N
		$
		is the analytic quotient by $\mu_N$, it follows that there exists a holomorphic map
		$
		g:\C^*\to\C^*
		$
		such that
		$
		\widetilde\psi(z)=z^m g(z^N).
		$
		Conversely, every map of this form satisfies
		$
		\widetilde\psi(\zeta z)=\zeta^m\widetilde\psi(z),
		$
		hence defines a strict holomorphic endomorphism.
		The statement about $2$-isomorphism follows from \Cref{prop:strictification}: two strict pairs with the same $\sigma$ are $2$-isomorphic exactly when their lifts differ by multiplication by a constant element of $\mu_N$. Since $m$ has been fixed in its chosen residue class, the same scalar multiplies $g$.
	\end{proof}
	
	\subsection{The elliptic quotient}
	
	Assume $\Gamma=q^\Z$ with $|q|\ne 1$.
	
	\begin{proposition}\label[proposition]{prop:elliptic}
		The action of $q^\Z$ on $\C^*$ is free and properly discontinuous, and the quotient
		\[
		E_q:=\C^*/q^\Z
		\]
		is a one-dimensional complex torus, hence an elliptic curve. In particular, $[\C^*/\Gamma]\simeq E_q$ as analytic stacks.
	\end{proposition}
	
	\begin{proof}
		Write $\C^*=\exp(\C)$ and choose a logarithm $\log q$.
		Then multiplication by $q$ corresponds to translation by $\log q$ on $\C$, together with the deck translation $2\pi i\Z$.
		Thus $E_q\simeq \C/(2\pi i\Z+\log q\,\Z)$.
		Free proper actions yield a quotient manifold and the quotient stack agrees with it.
	\end{proof}
	
	\begin{proposition}\label[proposition]{proposition:End-elliptic}
		Write $E_q=\C/\Lambda$ with $\Lambda=2\pi i\Z+\log q\,\Z$.
		Every holomorphic self-map of $E_q$ is affine:
		\[
		[w]\longmapsto[\alpha w+\beta],
		\]
		where $\beta\in \C/\Lambda$ and $\alpha\in\C$ satisfies $\alpha\Lambda\subset\Lambda$.
		The holomorphic group endomorphisms (those with $\beta=0$) form a ring equal to $\Z$ for generic $q$, and to an order in an imaginary quadratic field in the CM case. Translations, except the identity, are not group endomorphisms.
	\end{proposition}
	
	\begin{proof}
		Let $f:E_q\to E_q$ be holomorphic and let $\pi:\C\to E_q$ be the universal covering map.
		Since $\C$ is simply connected, $f\circ\pi$ lifts to a holomorphic map $F:\C\to\C$ such that
		$\pi\circ F=f\circ\pi$.
		For every $\lambda\in\Lambda$ and every $w\in\C$, the two points $w$ and $w+\lambda$ project to the same point of $E_q$, hence
		$F(w+\lambda)-F(w)\in\Lambda$.
		Because $\Lambda$ is discrete and the map $w\mapsto F(w+\lambda)-F(w)$ is holomorphic, it is constant.
		Thus there exists $c_\lambda\in\Lambda$ with
		$
		F(w+\lambda)=F(w)+c_\lambda.
		$
		Differentiating gives $F'(w+\lambda)=F'(w)$, so $F'$ is $\Lambda$-periodic and descends to a holomorphic function on the compact Riemann surface $E_q$.
		By the maximum principle, that descended function is constant; hence $F'(w)=\alpha$ for some $\alpha\in\C$ and therefore
		$
		F(w)=\alpha w+\beta
		$
		for some $\beta\in\C$.
		The relation $F(w+\lambda)-F(w)=c_\lambda\in\Lambda$ now becomes $\alpha\lambda\in\Lambda$ for all $\lambda\in\Lambda$, i.e.\ $\alpha\Lambda\subset\Lambda$.
		Conversely, every affine map $w\mapsto \alpha w+\beta$ with $\alpha\Lambda\subset\Lambda$ descends to a holomorphic self-map of $E_q$.
		The description of the endomorphism ring is the standard one for elliptic curves; see for instance \cite[Ch.~III, \S4]{Silverman09}.
	\end{proof}

	For $\Gamma=q^\Z$, a nonzero group endomorphism of $E_q=\C/\Lambda$ is an isogeny induced by multiplication by $\alpha$ on $\C$, with $\alpha\Lambda\subset\Lambda$.
	In the multiplicative coordinate $z=e^w$ on $\C^*$, it has the local expression $z\mapsto e^{\alpha\log z}$.
	When $\alpha\notin\Z$, the expression involves a choice of branch of $\log$ and therefore is not a single-valued holomorphic map $\C^*\to\C^*$.
	Elliptic isogenies which do not lift to the chosen atlas are described by Morita morphisms.
	The descended map on the elliptic curve is algebraic, although this expression need not define a map on the chosen atlas.

	Write $q=e^{\tau}$ for a fixed choice of logarithm $\tau\in\C$ with $\Re(\tau)\neq 0$.
	Then $E_q$ is uniformised by
	\[
	E_q \ \simeq\ \C/\Lambda,
	\qquad
	\Lambda:=2\pi i\,\Z+\tau\,\Z,
	\]
	via $w\mapsto e^{w}$.
	
	\begin{proposition}\label[proposition]{prop:meridian-in-qZ}
		Let $\eta\in \Omega^1_X(\log D)$ be closed on $U$ and let $\lambda_i=\exp(2\pi i\,\Res_{D_i}(\eta))$ be the meridian multipliers.
		If $\Gamma=\im(\rho)=q^\Z$, then for each $i$ there exists an integer $m_i\in\Z$ such that
		$
		\lambda_i=q^{m_i}.
		$
		Equivalently, for a fixed logarithm $\tau=\log q$ one has the congruence
		\begin{equation}\label{eq:residue-congruence}
			2\pi i\,\Res_{D_i}(\eta)\ \equiv\ m_i\,\tau \quad (\mathrm{mod}\ 2\pi i\,\Z),
		\end{equation}
		i.e.
		\[
		\Res_{D_i}(\eta)\ \in\ \frac{\tau}{2\pi i}\,\Z+\Z.
		\]
	\end{proposition}
	
	\begin{proof}
		By definition, $\lambda_i\in \Gamma=q^\Z$, hence $\lambda_i=q^{m_i}$ for some $m_i\in\Z$.
		Writing $q=e^{\tau}$ yields $\lambda_i=e^{m_i\tau}$.
		On the other hand $\lambda_i=e^{2\pi i\,\Res_{D_i}(\eta)}$, so equality of exponentials gives
		$2\pi i\,\Res_{D_i}(\eta)-m_i\tau\in 2\pi i\,\Z$, which is \eqref{eq:residue-congruence}.
	\end{proof}
	
	\begin{remark}
		Let $\Pi(\eta)\subset\C$ be the additive subgroup generated by the residue periods $2\pi i\,\Res_{D_i}(\eta)$ and the global periods
		$\int_{\delta_j}\eta$ (cf.\ \Cref{prop:period-residue}).
		If $\Gamma=q^\Z$, then necessarily
		\begin{equation*}
			\Pi(\eta)\subset \Lambda=2\pi i\,\Z+\tau\,\Z,
		\end{equation*}
		because $\exp(\Pi(\eta))=\Gamma$ and $\exp(w)=\exp(w')$ iff $w-w'\in 2\pi i\,\Z$.
		Thus $\Gamma=q^\Z$ is reflected in the containment of the periods of $\eta$ in a rank-two lattice in $\C$.
	\end{remark}
	
	\begin{lemma}\label[lemma]{lem:period-lattice-qZ}
		With the preceding notation, set
		$\Gamma=\exp(\Pi(\eta))\subset\mathbb C^*$.
		Then the following are equivalent:
		\begin{enumerate}[label=(\alph*),leftmargin=2.2em]
			\item $\Gamma$ is infinite cyclic and generated by some $q\in\C^*$ with $|q|\neq 1$, i.e.\ $\Gamma=q^\Z$ with $|q|\neq 1$.
			\item There exists $\tau\in\C$ with $\Re(\tau)\neq 0$ such that
			$
			\Pi(\eta)\subset 2\pi i\,\Z+\tau\,\Z
			$
			and the image of $\Pi(\eta)$ in $\C/(2\pi i\,\Z)$ is generated by the class of $\tau$.
			Equivalently, $\Pi(\eta)+2\pi i\,\Z=2\pi i\,\Z+\tau\,\Z$.
		\end{enumerate}
		In this case $q=e^{\tau}$ and $\Gamma=\exp(\Pi(\eta))=q^\Z$.
	\end{lemma}
	
	\begin{proof}
		(a)$\Rightarrow$(b): If $\Gamma=q^\Z$ with $q=e^\tau$ and $\Re(\tau)\neq 0$, then any $w\in\Pi(\eta)$ satisfies $e^w\in q^\Z$, hence
		$w\equiv n\tau\ (\mathrm{mod}\ 2\pi i\,\Z)$ for some $n\in\Z$. Thus $\Pi(\eta)\subset 2\pi i\,\Z+\tau\,\Z$, and the image of $\Pi(\eta)$ in
		$\C/(2\pi i\,\Z)$ is contained in $\tau\Z$ and contains $\tau$ up to a generator, so it is cyclic generated by $\tau$.
		(b)$\Rightarrow$(a): If $\Pi(\eta)\subset 2\pi i\,\Z+\tau\,\Z$ and its image modulo $2\pi i\,\Z$ equals $\tau\Z$,
		then $\exp(\Pi(\eta))\subset \exp(\tau\Z)=q^\Z$ with $q=e^\tau$.
		Conversely, since the image is all of $\tau\Z$, there exists $w\in\Pi(\eta)$ with $w\equiv \tau\ (\mathrm{mod}\ 2\pi i\,\Z)$,
		so $e^w=q$, and hence $q\in \Gamma$. Therefore $\Gamma=q^\Z$.
	\end{proof}
	
	\begin{corollary}\label[corollary]{cor:elliptic-dense-angular}
		Let $\Gamma=q^\Z$ with $|q|\neq 1$ and write $q=e^{\tau}$ with $\tau=a+ib$, $a\neq 0$.
		If $b/2\pi\notin\Q$, then the set of arguments $\{e^{i n b}:n\in\Z\}$ is dense in $\Sone$.
		Nevertheless $\Gamma$ is a discrete subset of $\C^*$, hence the action of $\Gamma$ on $\C^*$ is properly discontinuous.
	\end{corollary}
	
	\begin{proof}
		Density of $\{e^{inb}\}$ in $\Sone$ is the standard irrational-rotation fact.
		Discreteness of $\Gamma$ follows because $|q^n|=e^{an}$ tends to $0$ as $n\to-\infty$ and to $\infty$ as $n\to+\infty$;
		hence $\Gamma$ has no accumulation point in $\C^*$.
	\end{proof}
	
	\begin{corollary}\label[corollary]{cor:extra-unitary}
		Let $q\in\C^*$ with $|q|\neq 1$ and let $u\in\Sone$ have infinite order.
		Then the subgroup $\langle q,u\rangle\subset\C^*$ is not discrete (it has accumulation on the circle $|z|=1$).
		In particular, a discrete subgroup of $\C^*$ cannot contain a unitary element of infinite order; the remaining possibility is $\mu_N\cdot q^\Z$, as in \Cref{prop:discrete-subgroups}.
	\end{corollary}
	
	\begin{proof}
		The set $\{u^m:m\in\Z\}\subset\Sone$ has accumulation because it is infinite in a compact set.
		Thus $\langle q,u\rangle$ contains a non-discrete subset already on the unit circle, hence is not discrete in $\C^*$.
	\end{proof}

	\begin{proposition}\label[proposition]{prop:gcd-meridian}
		In the situation of \Cref{prop:meridian-in-qZ}, the subgroup of $\Gamma$ generated by the meridians equals $q^{d\Z}$ where
		$
		d:=\gcd(m_1,\dots,m_k).
		$
		In particular, the meridian subgroup is either trivial ($d=0$), all of $\Gamma$ ($d=1$), or a finite-index subgroup of $\Gamma$ ($d>1$).
	\end{proposition}
	
	\begin{proof}
		The subgroup generated by $\{q^{m_i}\}$ is $q^{\langle m_1,\dots,m_k\rangle}=q^{d\Z}$.
	\end{proof}

	For $\Gamma=q^\Z$, the same cyclicity constrains strict reparametrisations.
	
	\begin{proposition}\label[proposition]{prop:qZ-rigidity}
		Assume $\Gamma=q^\Z$ with $|q|\neq 1$.
		\begin{enumerate}[label=\textup{(\roman*)},leftmargin=2.6em]
			\item Let $\psi:[\C^*/\Gamma]\to[\C^*/\Gamma]$ be a stack endomorphism admitting a strict semi-equivariant representative
			$
			\widetilde\psi:\C^*\to\C^*
			$
			satisfying
			$
			\widetilde\psi(qz)=q^{n}\widetilde\psi(z)
			$
			for some $n\in\Z$.
			Then
			$
			\widetilde\psi(z)=c z^n
			$
			for some $c\in\C^*$.
			The induced endomorphism on $E_q$ is multiplication by $n$ composed with translation by the class of $c$.
			
			\item Let $R\in\C(z)$ be nonzero and let $\sigma:\Gamma\to\Gamma$ satisfy
			\[
			R(\gamma z)=\sigma(\gamma)R(z)
			\qquad
			(\gamma\in\Gamma,\ z\in\C^*).
			\]
			Then
			$
			R(z)=c z^n
			$
			for some $c\in\C^*$ and $n\in\Z$.
		\end{enumerate}
	\end{proposition}
	
	\begin{proof}
		Since $\Gamma=q^\Z$ is infinite cyclic, any homomorphism $\sigma:\Gamma\to\Gamma$ is determined by
		$
		\sigma(q)=q^n
		$
		for some $n\in\Z$.
		We first prove \textup{(ii)}. Let $\mathrm{Div}(R)$ be the divisor of zeros and poles of $R$ on $\PP^1$.
		The semi-equivariance relation implies that $\mathrm{Div}(R)$ is invariant under the action of $q^\Z$.
		If $a\in\C^*$ belonged to the support of $\mathrm{Div}(R)$, then the whole orbit
		$
		q^\Z\cdot a=\{q^m a:m\in\Z\}
		$
		would also lie in the support.
		This orbit is infinite because $q$ is not a root of unity, whereas the support of $\mathrm{Div}(R)$ is finite.
		Hence the support of $\mathrm{Div}(R)$ is contained in $\{0,\infty\}$, and therefore $R$ is a monomial:
		$
		R(z)=c z^n
		$
		for some $c\in\C^*$.
		For \textup{(i)}, set $h(z)=\widetilde\psi(z)/z^n$. The
		semi-equivariance identity gives $h(qz)=h(z)$. Thus $h$ descends to
		a holomorphic map $E_q=\C^*/q^\Z\to\C^*$. Since $E_q$ is
		compact, such a map is constant. Hence $\widetilde\psi(z)=c z^n$;
		no extension of $\widetilde\psi$ across $0$ or $\infty$ is needed.
		For $n\ne0$, the descended map on $E_q$ is the multiplication-by-$n$ isogeny followed by translation by the class of $c$; for $n=0$, it is constant.
	\end{proof}
	
	By \Cref{proposition:End-elliptic}, every nonconstant holomorphic self-map of $E_q$ is a translation composed with an isogeny; a constant map corresponds to $\alpha=0$.
	On the universal cover $\C$ such a map is affine, say $w\mapsto \alpha w+\beta$, with $\alpha\Lambda\subset\Lambda$.
	In multiplicative coordinates $z=e^w$, it is formally given by $z\mapsto e^{\alpha\log z}$ and is single-valued on $E_q$,
	but it need not be single-valued as a holomorphic map $\C^*\to\C^*$ unless $\alpha\in\Z$.
	Thus the strict atlas description captures only the monomial part, whereas the full elliptic endomorphism theory requires Morita, or bibundle, morphisms; see \cite{BehrendXu,Lerman,MoerdijkPronk}.
	
	Relative to a $\Z$-basis of $\Lambda$, the condition $\alpha\Lambda\subset\Lambda$ is encoded by an integer $2\times 2$ matrix.
	On the abstract group $\Gamma=q^\Z$ one has $\Aut(\Gamma)=\{\pm 1\}$. By contrast, the elliptic curve has translations as well as group endomorphisms arising from multiplication by elements preserving its period lattice.
	
	\begin{theorem}\label{thm:EndP1-regimes}
		Let $\Gamma\subset\C^*$ act on $\PP^1$ by multiplication, fixing
		$0$ and $\infty$. A nonconstant strict semi-equivariant rational map
		is specified by $R\in\C(z)$ and a homomorphism
		$\sigma:\Gamma\to\Gamma$ such that
		$R(\gamma z)=\sigma(\gamma)R(z)$ for every $\gamma\in\Gamma$.
		The possibilities are as follows.
		\begin{enumerate}[label=\textup{(\alph*)},leftmargin=2.2em]
			\item If $\Gamma=\mu_N$ is finite, then
			\[
			R(z)=z^m\frac{P(z^N)}{Q(z^N)},\qquad
			\sigma(\zeta)=\zeta^m,
			\]
			where $P,Q\in\C[t]$ are coprime and the displayed rational map is
			nonconstant.
			\item If $\Gamma$ is infinite, then
			$R(z)=cz^n$, where $c\in\C^*$ and
			$n\in\Z\setminus\{0\}$, and $\sigma(\gamma)=\gamma^n$.
	\end{enumerate}\end{theorem}
	Constant maps are excluded. A constant map to $0$ or $\infty$
	is semi-equivariant for every $\sigma$, whereas a constant map with
	value in $\C^*$ requires $\sigma$ to be trivial. This classification
	concerns the fixed action-groupoid presentation, not arbitrary
	morphisms of $[\PP^1/\Gamma]$.

	\begin{proof}
		If $\Gamma=\mu_N$, every homomorphism
		$\sigma:\mu_N\to\mu_N$ has the form
		$\sigma(\zeta)=\zeta^m$ for some $m\in\Z$. Hence the
		semi-equivariance relation says that $R(z)/z^m$ is invariant
		under $z\mapsto\zeta z$. Since
		$\C(z)^{\mu_N}=\C(z^N)$, it follows that
		\[
		R(z)=z^m\frac{P(z^N)}{Q(z^N)}
		\]
		for coprime polynomials $P,Q\in\C[t]$. Conversely, every such
		nonconstant rational function is semi-equivariant with respect to
		$\zeta\mapsto\zeta^m$.
		For infinite $\Gamma$ the assertion follows from
		\Cref{proposition:mero-rigidity}.
	\end{proof}

	\section{Automorphisms of multiplicative quotient stacks}\label{sec:AutBGamma}
	
	Self-equivalences of an analytic stack and the $2$-isomorphisms
	between them form an automorphism $2$-group
	\cite{NoohiButterflies,NoohiWP2}. Its group of isomorphism classes is
	$\pi_0\Aut_2(\mathcal X)$, while
	\[
	\pi_1\Aut_2(\mathcal X)
	=\Aut_2(\mathrm{id}_{\mathcal X})
	\]
	is the group of $2$-automorphisms of the identity. The latter is
	determined by global inertia, not merely by the existence of points
	with nontrivial stabilisers. For the multiplicative quotients considered
	here, $\pi_1$ is trivial even when the compactified quotient has
	nontrivial inertia at $0$ and $\infty$.
	
	We first describe the strict automorphisms of the chosen action
	groupoid. Intrinsic automorphisms of the quotient stack need not admit
	such representatives; the elliptic quotient gives a familiar example.
	
	\subsection{Quotient stacks and their automorphism 2-groups}\label{subsec:UG-automorphisms}
	
	For an analytic quotient stack $\mathcal X=[U/G]$, intrinsic automorphisms are Morita self-equivalences, which need not have strict representatives on the chosen action groupoid. This is the standard point of view in the theory of quotient stacks and automorphism $2$-groups; see \cite{Noohi05,NoohiButterflies,Romagny}.
	
	Let $G$ be a complex Lie group acting holomorphically on a complex manifold $U$, and consider the action groupoid
	\[
	U\times G \rightrightarrows U,
	\qquad
	s(u,g)=u,
	\qquad
	t(u,g)=g\cdot u.
	\]
	A strict automorphism of this groupoid is determined by a pair $(f,\alpha)$, where
	\[
	f\in \Aut_{\mathrm{hol}}(U),
	\qquad
	\alpha\in \Aut(G),
	\]
	subject to the semi-equivariance relation
	$
	f(g\cdot u)=\alpha(g)\cdot f(u)
	$ for all 
	$g\in G,\ u\in U.
	$
	If $(f,\alpha)$ and $(f',\alpha')$ are two such strict pairs, a $2$-isomorphism between them is given by a holomorphic map
	$
	b:U\to G
	$
	satisfying
	\[
	f'(u)=b(u)\cdot f(u),
	\qquad
	b(g\cdot u)\,\alpha(g)=\alpha'(g)\,b(u)
	\qquad
	\text{for all }g\in G,\ u\in U.
	\]
	Thus two strict models may define the same automorphism of the quotient stack even when they are not literally equal. The automorphism $2$-group has two invariants:
	\[
	\pi_0\,\Aut_2([U/G]) \ \ \ \mathrm{and} \ \ \ \pi_1\,\Aut_2([U/G]).
	\]
	Concretely, in the quotient model, the right group consists of holomorphic maps
	$
	b:U\to G
	$
	such that $b(u)\in \mathrm{Stab}(u)$ for every $u$, together with the evident equivariance condition. In particular, if the action of $G$ on $U$ is free, then
	$
	\pi_1\,\Aut_2([U/G])=1.
	$
	
	An intrinsic self-equivalence need not admit a strict representative on the chosen atlas. It is instead represented by a Morita self-bibundle of the action groupoid \cite{BehrendXu,Noohi05}. The presenting groupoid and its quotient stack are related by
	\[
	\begin{tikzcd}[column sep=large,row sep=large]
		U\times G \arrow[r, shift left=1.2ex, "s"] \arrow[r, shift right=1.2ex, "t"'] &
		U \arrow[r] &
		{[U/G]}.
	\end{tikzcd}
	\]
	
	The quotient stack retains the group-action data even when the coarse orbit space is non-Hausdorff. For the compactified presentation $[\PP^1/\Gamma]$, it also retains the stabilisers of $0$ and $\infty$.
	
	\subsection{Inertia and 2-automorphisms of the quotient bases}\label{subsec:inertia-quotient-bases}
	
	For an analytic stack $\mathcal{X}$, the group $\pi_1 \Aut_2(\mathcal{X})$ is governed by the inertia stack; see \cite[\S3.4]{Noohi05}. More precisely, if $x:W\to \mathcal X$ is a $W$-valued point, its inertia group is the automorphism group of $x$ in the groupoid $\mathcal X(W)$. The inertia stack $I\mathcal X$ is the stack whose $W$-points are pairs $(x,\alpha)$, where $x\in \mathcal X(W)$ and $\alpha$ is an automorphism of $x$. A $2$-automorphism of the identity functor is a global section of the inertia stack $I\mathcal X$. Equivalently, one has a canonical identification
	\[
	\pi_1\,\Aut_2(\mathcal X)\cong \Gamma(\mathcal X,I\mathcal X).
	\]
	
	For the multiplicative quotients the global inertia can be computed
	directly.
	
	\begin{proposition}\label[proposition]{prop:inertia-BGamma}
		Let $\Gamma\subset\C^*$ be regarded as a discrete group acting by multiplication. The action on $\C^*$ is free, and the only points of $\PP^1$ with nontrivial stabiliser are $0$ and $\infty$. Nevertheless, the identity functor of either quotient stack has no nontrivial global $2$-automorphisms:
		\[
		\pi_1\Aut_2([\C^*/\Gamma])=1,\qquad
		\pi_1\Aut_2([\PP^1/\Gamma])=1.
		\]
	\end{proposition}
	
	\begin{proof}
		A global $2$-automorphism of the identity is represented on the atlas by a holomorphic map $a:Y\to\Gamma$ with $a(y)y=y$. For $Y=\C^*$ the free action forces $a=1$. For $Y=\PP^1$ the same conclusion holds on the dense open subset $\C^*$; since $\Gamma$ is discrete and $\PP^1$ is connected, $a$ is constant and therefore trivial throughout. The presence of nontrivial stabilisers over $0$ and $\infty$ does not supply independent global sections of inertia.
	\end{proof}
	
	Thus the automorphism $2$-group of the open quotient has trivial $\pi_1$; its group of isomorphism classes is $\pi_0\Aut_2([\C^*/\Gamma])$.
	The compactified quotient has nontrivial point stabilisers at the boundary, although these do not give global $2$-automorphisms of the identity.
	
	\subsection{Automorphisms of the open base}
	
	By \Cref{prop:inertia-BGamma}, the automorphism $2$-group of the open base has trivial $\pi_1$. We first classify the strict self-equivalences of its fixed presentation.
	
	\begin{theorem}\label{thm:AutB}
		Let $\Gamma\subset\C^*$ be regarded as a discrete group. The isomorphism classes of strict automorphisms of the action groupoid $\C^*\rtimes\Gamma$ form the group
		\[
		(\C^*/\Gamma)\rtimes(\Z/2).
		\]
		They are represented by $z\mapsto az$ and $z\mapsto a/z$, with parameters identified modulo $\Gamma$. In the proper cyclic case $\Gamma=q^\Z$, the intrinsic automorphism group of the elliptic quotient $E_q=\C^*/q^\Z$ is instead $E_q\rtimes\Aut(E_q,0)$; it can contain additional automorphisms which do not lift to this atlas.
	\end{theorem}
	
	\begin{proof}
		A strict groupoid automorphism has a biholomorphic object map $\C^*\to\C^*$, hence is $z\mapsto az$ or $z\mapsto a/z$ by \Cref{prop:hol-normal-aut-Cstar}. Their semi-equivariance homomorphisms are the identity and inversion, respectively. \Cref{prop:strictification} identifies two parameters precisely when they differ by an element of $\Gamma$. For the proper cyclic action the quotient is the elliptic curve $E_q$; the usual classification of automorphisms of a complex torus gives $E_q\rtimes\Aut(E_q,0)$, without requiring a lift to $\C^*$.
	\end{proof}
	
	For strict self-equivalences the induced map on $\Gamma$ is the identity or inversion. This does not describe all Morita self-equivalences of a proper quotient.
	
	\subsection{Automorphisms of the compactified base}\label{sec:2group}
	
	The compactified presentation has nontrivial stabilisers at $0$ and $\infty$. We describe here its strict self-equivalences, keeping them distinct from general Morita self-equivalences.
	
	\begin{theorem}\label{thm:2group}
		Let $\Gamma\ne\{1\}$. The strict automorphisms of the action groupoid $\PP^1\rtimes\Gamma$ which preserve $\{0,\infty\}$ on the atlas have isomorphism classes
		\[
		(\C^*/\Gamma)\rtimes(\Z/2),
		\]
		represented by $z\mapsto az$ and $z\mapsto a/z$. The automorphism group of the identity functor of the quotient stack is trivial.
	\end{theorem}
	
	\begin{proof}
		A Möbius transformation preserving the unordered pair $\{0,\infty\}$ is a scaling or inversion. The same semi-equivariance calculation as in \Cref{thm:AutB} describes strict classes up to natural isomorphism. The assertion about global $2$-automorphisms is \Cref{prop:inertia-BGamma}.
	\end{proof}
	
	\begin{remark}
		In the case $\Gamma=\{1\}$, one recovers the familiar automorphism groups
		$
		\Aut(\mathbb C^*)\cong \mathbb C^*\rtimes(\mathbb Z/2)$ and $
		\Aut(\mathbb P^1)=\PGL_2(\mathbb C),$
		and there is no nontrivial $\pi_1$.
	\end{remark}

	\section{The logarithmic boundary and its angular foliations}\label{sec:KN}

	We recall the local description of $X^{\log}$ as a real oriented blow-up of the normal crossings divisor and compute the boundary characters of $\eta$. The resulting constant-coefficient angular distributions on the boundary tori are then described by their leaves and closures.
	
	The discussion uses the standard construction of Kato and Nakayama \cite{KN99}; for comparisons with root-stack constructions and their role in logarithmic topology, see also \cite{CSSST17,VistoliTalpo18}.
	
	\subsection{The Kato--Nakayama space and the boundary characters}
	
	Let $j:U\hookrightarrow X$ be the inclusion.
	
	\begin{definition}\label[definition]{def:log-structure}
		The standard fine saturated log structure associated with the divisor $D$ is
		$
		\mathcal M_X:=\mathcal O_X\cap j_*\mathcal O_U^*,
		$
		with structure morphism $\alpha:\mathcal M_X\to\mathcal O_X$ given by inclusion. This is the usual logarithmic structure attached to a normal crossings boundary.
	\end{definition}
	
	The above definition means that the local sections of $\mathcal M_X$ are holomorphic functions that do not vanish on the complement of $D$.
	
	\begin{definition}\label[definition]{def:KN}
		The Kato--Nakayama space $X^{\log}$ is the topological space whose points are pairs $(x,h)$, where $x\in X$ and
		$
		h:\mathcal M_{X,x}^{\mathrm{gp}}\to \mathbb S^1
		$
		is a group homomorphism satisfying
		$
		h(u)=u(x)/|u(x)|
		$
		for every unit $u\in \mathcal O_{X,x}^*\subset \mathcal M_{X,x}$.
		It comes equipped with the continuous projection
		\[
		\pi:X^{\log}\to X,
		\qquad
		\pi(x,h)=x.
		\]
		This is the Kato--Nakayama realisation of the fine and saturated log space $(X,\mathcal M_X)$, see \cite{KN99}.
	\end{definition}
	
	Let $x\in X$ be a point at which exactly $k$ irreducible components of $D$ meet, and choose local coordinates
	$(z_1,\dots,z_n)$ on a neighbourhood $U_x\subset X$ such that
	\[
	D\cap U_x=\{z_1\cdots z_k=0\},
	\qquad
	D_i\cap U_x=\{z_i=0\}
	\quad (1\le i\le k).
	\]
	In this simple normal crossings situation, the Kato--Nakayama space admits the familiar local description as a real oriented blow-up of the boundary. More precisely, the chosen local coordinates give a homeomorphism
	\begin{equation}\label{eq:KN-local-model}
		\pi^{-1}(U_x)\simeq (\mathbb R_{\ge 0}\times \mathbb S^1)^k\times \mathbb C^{n-k},
	\end{equation}
	under which the projection $\pi$ is given by
	\[
	(r_1,\theta_1,\dots,r_k,\theta_k,z_{k+1},\dots,z_n)
	\longmapsto
	(r_1e^{i\theta_1},\dots,r_ke^{i\theta_k},z_{k+1},\dots,z_n).
	\]
	Thus $\pi^{-1}(U_x\setminus D)\to U_x\setminus D$ is a homeomorphism. The fibre over $x\in D_I^\circ$ is naturally a torsor under $\operatorname{Hom}(\overline{\mathcal M}_{X,x}^{\mathrm{gp}},\mathbb S^1)\simeq(\mathbb S^1)^{|I|}$, where $\overline{\mathcal M}_{X,x}=\mathcal M_{X,x}/\mathcal O_{X,x}^*$; a choice of local boundary parameters identifies it with a torus, but does not give it a canonical origin.
	Indeed, in the chosen chart the logarithmic structure is generated by $z_1,\dots,z_k$ together with the units of $\mathcal O_X$. A point of $X^{\log}$ over $U_x$ is therefore specified by a point of $U_x$ together with a choice of argument for each local boundary parameter $z_i$. Writing
	\[
	z_i=r_ie^{i\theta_i},
	\qquad
	r_i\ge 0,
	\qquad
	\theta_i\in \mathbb R/2\pi\mathbb Z,
	\]
	one sees that the compatibility condition with units determines the phase whenever $z_i\neq 0$, whereas if $z_i=0$ the angular variable remains free. This yields exactly the model \eqref{eq:KN-local-model}. See \cite[\S2]{KN99} for the construction of $X^{\log}$ and its local model; compare also \cite[\S3--\S4]{CSSST17} and \cite[\S2]{VistoliTalpo18}.
	
	In the local model \eqref{eq:KN-local-model}, one has, for each $1\le i\le k$,
	\begin{equation}\label{eq:dz-over-z-kn}
		\pi^*\!\left(\frac{dz_i}{z_i}\right)=\frac{dr_i}{r_i}+i\,d\theta_i
		\qquad
		\text{on }\pi^{-1}(U_x\setminus D).
	\end{equation}
	Indeed, on $\pi^{-1}(U_x\setminus D)$ one may write $z_i=r_ie^{i\theta_i}$, so that
	$
	\log z_i=\log r_i+i\theta_i
	$
	after choosing a local branch upstairs. Differentiating gives
	\[
	\frac{dz_i}{z_i}=d(\log r_i+i\theta_i)=\frac{dr_i}{r_i}+i\,d\theta_i,
	\]
	and pulling back by $\pi$ yields \eqref{eq:dz-over-z-kn}. This is the standard decomposition of a logarithmic differential into its radial and angular parts on $X^{\log}$.
	
	Assume now that the transversely affine structure is logarithmic along $D$, so that $\eta\in \Omega_X^1(\log D)$ and $d\eta=0$ on $U$. In the above snc chart one may write
	\begin{equation}\label{eq:eta-local}
		\eta=\sum_{i=1}^k \alpha_i\frac{dz_i}{z_i}+\beta,
		\qquad
		\alpha_i\in\mathbb C,
		\qquad
		\beta\ \text{holomorphic and closed}.
	\end{equation}
	Substituting \eqref{eq:dz-over-z-kn} into \eqref{eq:eta-local}, one obtains on $\pi^{-1}(U_x\setminus D)$
	\begin{equation}\label{eq:eta-pullback-kn}
		\pi^*\eta=\sum_{i=1}^k \alpha_i\left(\frac{dr_i}{r_i}+i\,d\theta_i\right)+\pi^*\beta.
	\end{equation}
	
	\begin{proposition}\label[proposition]{prop:boundary-fiber-restriction}
		Let $x\in D_I^\circ$ and write in an snc chart
		\[
		\eta=\sum_{i=1}^k\alpha_i\frac{dz_i}{z_i}+\beta,
		\qquad \beta\ \text{holomorphic}.
		\]
		The angular component of this logarithmic form on the boundary fibre
		$T_{I,x}=\pi^{-1}(x)\simeq(S^1)^{|I|}$ is the well-defined closed form
		\[
		\eta^{\mathrm{ang}}_{I,x}:=i\sum_{i\in I}\alpha_i\,d\theta_i.
		\]
		It is independent of the choice of local defining equations for the components of $D$.
	\end{proposition}
	
	\begin{proof}
		The polar decomposition of $dz_i/z_i$ on the interior is $dr_i/r_i+i\,d\theta_i$. The angular component of its logarithmic residue is therefore $i\,d\theta_i$ for $i\in I$. The holomorphic remainder has no angular contribution on the fibre. A change of defining equation $z_i'=u_i z_i$, where $u_i$ is a unit, changes the angle by the argument of $u_i$; this is constant on the fibre over $x$, and hence leaves $d\theta_i$ unchanged. The formula follows.
	\end{proof}
	
	For a boundary stratum $D_I^\circ$, let $\partial_I X^{\log}:=\pi^{-1}(D_I^\circ)$. By \Cref{prop:boundary-fiber-restriction}, the angular residue form on the boundary fibre
	$
	\pi^{-1}(x)\simeq (\mathbb S^1)^{|I|}
	$
	is $\eta^{\mathrm{ang}}_{I,x}=i\sum_{i\in I}\alpha_i\,d\theta_i$. 
	Accordingly, the canonical object attached to the boundary fibre is not, in general, a single-valued function on $(\mathbb S^1)^{|I|}$, but the corresponding monodromy character
	\begin{equation}\label{eq:boundary-character}
		\chi_{I,x}:\pi_1\bigl(\pi^{-1}(x)\bigr)\cong \Z^{|I|}\longrightarrow \C^*,
	\end{equation}
	defined by
	\begin{equation*}
		\chi_{I,x}(m_i)_{i\in I}
		=
		\exp\!\left(2\pi i\sum_{i\in I}\alpha_i m_i\right).
	\end{equation*}
	Equivalently, on the universal cover $\R^{|I|}\to (\mathbb S^1)^{|I|}$ with coordinates $t_i$, one has the $\C^*$-valued function
	\[
	\widetilde\chi_{I,x}(t_i)_{i\in I}
	=
	\exp\!\left(2\pi i\sum_{i\in I}\alpha_i t_i\right),
	\]
	whose deck-transformation law is governed by \eqref{eq:boundary-character}. The image of $\chi_{I,x}$ is the subgroup of $\C^*$ generated by $\exp(2\pi i\alpha_i)$ for $i\in I$.
	Indeed, if $e_i\in \pi_1((\mathbb S^1)^{|I|})\cong \Z^{|I|}$ denotes the $i$-th coordinate loop, then
	\[
	\int_{e_i}\eta^{\mathrm{ang}}_{I,x}=2\pi i\,\alpha_i,
	\]
	and exponentiation yields
	$
	\chi_{I,x}(e_i)=\exp(2\pi i\alpha_i).
	$
	This is the classical residue--monodromy rule for logarithmic regular-singular rank-one connections, now read directly on the boundary torus of $X^{\log}$; see \cite[\S I.3]{Deligne70}.
	
	\subsection{Real foliations and angular dynamics on the boundary tori}
	
	We next study the foliation defined by $\pi^*\eta$ on the interior of $X^{\log}$ and the angular distributions defined by its residues on the boundary tori.
	
	\begin{proposition}\label[proposition]{prop:real-codim2}
		Set
		$
		\eta_1:=\Re(\pi^*\eta), \eta_2:=\Im(\pi^*\eta)
		$
		as real $1$-forms on $\pi^{-1}(U)$. Then $\eta_1$ and $\eta_2$ are closed, and the distribution
		\begin{equation}\label{eq:real-codim2}
			\mathcal D:=\ker(\eta_1)\cap\ker(\eta_2)\subset T_{\pi^{-1}(U)}
		\end{equation}
		is integrable on its constant-rank locus. It has real codimension two wherever $\eta\ne0$. On a simply connected open subset $W\subset\pi^{-1}(U)$, choose a primitive $\Phi$ of $\pi^*\eta$. The leaves of $\mathcal D$ in $W$ are the connected components of the level sets of $\Phi$.
	\end{proposition}
	
	\begin{proof}
		Since $d\eta=0$ on $U$, one has
		$
		d(\pi^*\eta)=\pi^*(d\eta)=0
		$
		on $\pi^{-1}(U)$. Hence both $\eta_1$ and $\eta_2$ are closed, and Frobenius integrability follows immediately. If
		$
		\pi^*\eta=d\Phi
		$
		on $W$, then
		\[
		\eta_1=d(\Re\Phi)
		\qquad\text{and}\qquad
		\eta_2=d(\Im\Phi).
		\]
		Therefore $\mathcal D$ is tangent to the common level sets of $\Re\Phi$ and $\Im\Phi$, equivalently to the level sets of $\Phi$.
	\end{proof}
	
	Work in an snc chart as above, and write
	\begin{equation*}
		\eta=\sum_{i=1}^k \alpha_i \frac{dz_i}{z_i}+\beta,
		\qquad
		\alpha_i\in\mathbb C,\;
		\beta\ \text{holomorphic and closed}.
	\end{equation*}
	On the region where $r_i>0$ for all $i$, set $x_i:=\log r_i$, choose lifts $\theta_i\in\mathbb R$, and define
	$
	w_i:=x_i+i\theta_i .
	$
	After shrinking to a simply connected open set if necessary, write $\beta=dh$ with $h$ holomorphic.
	
	\begin{proposition} \label[proposition]{prop:affine-leaf}
		On the universal cover of the real coordinates $(x_1,\theta_1,\dots,x_k,\theta_k)$, the distribution~\eqref{eq:real-codim2} is tangent to the level sets of the holomorphic function
		\begin{equation}\label{eq:affine-leaf-function}
			\widetilde F(w_1,\dots,w_k):=\sum_{i=1}^k \alpha_i w_i+h .
		\end{equation}
		Equivalently, the leaves are the connected components of the sets
		$
		\Re(\widetilde F)=\mathrm{constant},
		$ and $
		\Im(\widetilde F)=\mathrm{constant}.
		$
		If $\beta=0$, this reduces to the complex affine function
		$
		F(w_1,\dots,w_k)=\sum_{i=1}^k \alpha_i w_i .
		$
	\end{proposition}
	
	\begin{proof}
		By \eqref{eq:dz-over-z-kn}, one has
		$
		dw_i=d(\log r_i+i\theta_i)=dr_i/r_i+i\,d\theta_i .
		$
		Hence
		\[
		d\widetilde F
		=
		\sum_{i=1}^k \alpha_i\,dw_i+dh
		=
		\sum_{i=1}^k \alpha_i\Bigl(\frac{dr_i}{r_i}+i\,d\theta_i\Bigr)+\beta
		=
		\pi^*\eta
		\]
		by \eqref{eq:eta-pullback-kn}. Therefore
		$
		d(\Re\widetilde F)=\Re(\pi^*\eta),
		$ and $
		d(\Im\widetilde F)=\Im(\pi^*\eta).
		$
		Since $\mathcal D=\ker(\Re(\pi^*\eta))\cap\ker(\Im(\pi^*\eta))$, it follows that $\mathcal D$ is tangent exactly to the common level sets of $\Re(\widetilde F)$ and $\Im(\widetilde F)$, equivalently to the level sets of $\widetilde F$.
	\end{proof}

	\begin{proposition}\label[proposition]{prop:explicit-real-leaf-equations}
		Keep the local notation above and write $
		\alpha_i=a_i+ib_i,
		$ with $
		a_i,b_i\in\R.$
		Then, on the universal cover of the angular variables, the leaves of $\mathcal D$ are the connected components of the common level sets of the two real-valued functions
		\begin{equation}\label{eq:Phi1-Phi2-local}
			\Phi_1=
			\sum_{i=1}^k a_i \log r_i
			-
			\sum_{i=1}^k b_i \theta_i
			+
			\Re(h),
			\qquad
			\Phi_2=
			\sum_{i=1}^k b_i \log r_i
			+
			\sum_{i=1}^k a_i \theta_i
			+
			\Im(h).
		\end{equation}
		Equivalently,
		$
		\mathcal D=\ker(d\Phi_1)\cap \ker(d\Phi_2).$
		The residues determine the logarithmic normal terms; the holomorphic remainder contributes to the local level sets.
	\end{proposition}
	
	\begin{proof}
		Writing $w_i=\log r_i+i\theta_i$ and $\alpha_i=a_i+ib_i$, one computes
		\[
		\alpha_i w_i
		=
		(a_i+ib_i)(\log r_i+i\theta_i)
		=
		(a_i\log r_i-b_i\theta_i)
		+
		i(b_i\log r_i+a_i\theta_i).
		\]
		Taking real and imaginary parts in \eqref{eq:affine-leaf-function} therefore yields exactly the formulas in \eqref{eq:Phi1-Phi2-local}. The identity
		$
		\mathcal D=\ker(d\Phi_1)\cap \ker(d\Phi_2)
		$
		is just \Cref{prop:affine-leaf} rewritten in real coordinates. The final statement follows from the fact that the coefficients of the logarithmic variables $\log r_i$ and $\theta_i$ are      the real and imaginary parts of the residues.
	\end{proof}

	At every regular point, a leaf of $\mathcal D$ is a smooth real submanifold of codimension $2$, and is therefore locally diffeomorphic to $\R^{2n-2}$. Indeed, on the regular locus the forms $d\Phi_1$ and $d\Phi_2$ are linearly independent, so the claim follows immediately from the implicit function theorem applied to \eqref{eq:Phi1-Phi2-local}.
	\Cref{prop:explicit-real-leaf-equations} is useful not only for its explicit formulas, but also for the geometric distinction it makes clear. The residues determine the logarithmic normal and angular behaviour of the leaves near the boundary, namely the meridional winding, the radial dilation, and the induced linear foliation on the boundary tori. But they do not determine the global diffeomorphism type of a full leaf of $\mathcal D$. The remaining local regular contribution is encoded by the holomorphic term $h$, while the global geometry of the leaves also depends on the periods and on the topology of $U$.

	Let $D_I^\circ$ be a stratum along which exactly $\ell=|I|$ components of $D$ meet. Writing $\alpha_i=a_i+ib_i$ with $a_i,b_i\in\R$, the real and imaginary parts of the angular residue form are
	\[
	\eta^{\mathrm{ang}}_{1,I}:=-\sum_{i\in I}b_i\,d\theta_i,
	\qquad
	\eta^{\mathrm{ang}}_{2,I}:=\sum_{i\in I}a_i\,d\theta_i.
	\]
	Their common kernel defines the linear angular distribution
	\[
	\mathcal D_{I,x}:=\ker\eta^{\mathrm{ang}}_{1,I}\cap\ker\eta^{\mathrm{ang}}_{2,I}
	\quad\text{on }T_{I,x}.
	\]
	This is determined by the residue vector $(\alpha_i)_{i\in I}$; it is distinct from the real distribution $\mathcal D$ defined on the interior.
	
	For $\ell=2$, the classification is elementary. Let $T^2$ be a boundary fibre with angular coordinates $(\theta_1,\theta_2)$ and write the corresponding residues as $\alpha_i=a_i+ib_i$. Consider the determinant
	\[
	\Delta:=
	\det
	\begin{pmatrix}
		a_1 & a_2\\
		b_1 & b_2
	\end{pmatrix}
	=
	a_1b_2-a_2b_1.
	\]
	
	\begin{proposition}\label[proposition]{prop:T2-classification}
		Let $T^2$ be a boundary fibre with angular coordinates $(\theta_1,\theta_2)$ and residues $\alpha_i=a_i+ib_i$.
		\begin{enumerate}[label=\textup{(\alph*)},leftmargin=2.4em]
			\item If $\Delta\neq 0$, then the two angular forms are linearly independent. Consequently,
			\(
			\mathcal D_{I,x}=\{0\},
			\)
			so each angular leaf is a point.
			
			\item If $\Delta=0$ and $(a_1,a_2)\neq (0,0)$, then the induced foliation on $T^2$ is the linear foliation defined by
			\[
			a_1\,d\theta_1+a_2\,d\theta_2=0.
			\]
			Its leaves are closed circles if the line $\{a_1v_1+a_2v_2=0\}$ is rational, and dense in $T^2$ otherwise.
			
			\item If $\Delta=0$, $(a_1,a_2)=(0,0)$, and $(b_1,b_2)\neq (0,0)$, then the induced foliation on $T^2$ is defined by
			\[
			b_1\,d\theta_1+b_2\,d\theta_2=0.
			\]
			Again, the leaves are closed precisely in the rational-slope case and dense in the irrational-slope case.
			\item If $a_1=a_2=b_1=b_2=0$, the distribution is the full tangent bundle of $T^2$, whose only leaf is $T^2$.
		\end{enumerate}
	\end{proposition}
	
	\begin{proof}
		This is the standard Kronecker-foliation dichotomy on the two-torus; we record the short argument in the present notation.
		On $T^2$, the angular residue forms are constant-coefficient real $1$-forms. Their coefficient vectors are linearly independent if and only if $\Delta\neq 0$. In that case their common kernel in $\mathbb R^2$ is trivial, so $\mathcal D_{I,x}$ vanishes identically and its leaves are points.
		Assume now that $\Delta=0$. Then the two covectors are proportional, so their common kernel is one-dimensional, provided they do not both vanish. If $(a_1,a_2)\neq(0,0)$, the foliation is given by the single linear equation
		\[
		a_1\,d\theta_1+a_2\,d\theta_2=0.
		\]
		If instead $(a_1,a_2)=(0,0)$ and $(b_1,b_2)\neq(0,0)$, it is given by
		\[
		b_1\,d\theta_1+b_2\,d\theta_2=0.
		\]
		If all residues vanish, both forms vanish and $T^2$ is the only leaf. Otherwise one obtains a linear foliation on the real two-torus. On the universal cover $\R^2$, the leaves are affine lines parallel to $(a_2,-a_1)$, respectively $(b_2,-b_1)$. The projected line is closed precisely when its direction is rational, that is, spanned by a nonzero integer vector; otherwise it is dense in $T^2$. This gives the stated dichotomy.
	\end{proof}
	
	Let $T^\ell=(\mathbb S^1)^\ell$ be a boundary torus, and write $\mathbb R^\ell\to T^\ell$ for its universal covering.
	The angular distribution on $T^\ell$ is the kernel of the two constant covectors determined by the real and imaginary parts of the residue vector. Accordingly, if
	$a=(a_i)_{i\in I}$ and $b=(b_i)_{i\in I}$, we set
	\begin{equation}\label{eq:boundary-linear-space}
		V:=\{v\in \mathbb R^\ell \mid a\cdot v=0,\ b\cdot v=0\}\subset \mathbb R^\ell .
	\end{equation}
	The leaves on the universal cover are then translates of the linear subspace $V$.
	
	\begin{proposition}\label[proposition]{prop:closure-subtorus}
		Let $L\subset T^\ell$ be a leaf of the angular distribution on a boundary torus $T^\ell$.
		Then the closure $\overline L$ is a translate of a subtorus of $T^\ell$.
		More precisely, let $V_{\mathbb Q}\subset \mathbb Q^\ell$ be the smallest $\mathbb Q$-subspace such that
		$V\subset V_{\mathbb Q}\otimes_{\mathbb Q}\mathbb R$. Then $\overline L$ is a translate of the subtorus whose tangent space is
		$V_{\mathbb Q}\otimes_{\mathbb Q}\mathbb R$.
		In particular, $L$ is closed if and only if $V$ is defined over $\mathbb Q$.
	\end{proposition}
	
	\begin{proof}
		This is a standard fact on linear subgroups of compact tori; we record the short proof in the present notation.
		On the universal cover $\mathbb R^\ell$, the leaves are the affine subspaces parallel to $V$.
		Their images in $T^\ell=\mathbb R^\ell/\mathbb Z^\ell$ are therefore translates of the image of $V$, that is, of the subgroup
		\[
		\exp(2\pi iV)\subset T^\ell.
		\]
		The closure of a connected subgroup of a compact torus is again a compact connected subgroup, hence a subtorus. Its Lie algebra is the smallest real subspace defined over $\mathbb Q$ that contains $V$, namely $V_{\mathbb Q}\otimes_{\mathbb Q}\mathbb R$.
		The final assertion is the standard closedness criterion for linear subgroups of a torus: the image of $V$ in $T^\ell$ is closed if and only if
		$V\cap \mathbb Z^\ell$ is a full lattice in $V$, equivalently if and only if $V$ is defined over $\mathbb Q$.
	\end{proof}

	\begin{theorem}\label[theorem]{thm:topology-angular-leaf}
		Let $T^\ell=(\mathbb S^1)^\ell$ be a boundary torus over a stratum $D_I^\circ$, and let
		\[
		V=\{v\in \R^\ell\mid a\cdot v=0,\ b\cdot v=0\}
		\]
		be the linear subspace defined in \eqref{eq:boundary-linear-space}. Write
		\[
		r_I:=\rk_{\R}\langle a,b\rangle\in\{0,1,2\},
		\qquad
		\Lambda_I:=(2\pi)\Z^\ell,
		\qquad
		s_I:=\rk_{\Z}(V\cap \Lambda_I).
		\]
		Then every leaf $L\subset T^\ell$ of the angular distribution, with its immersed-leaf topology, is diffeomorphic to
		\begin{equation}\label{eq:leaf-topology-boundary}
			L\cong V/(V\cap \Lambda_I)\cong \R^{\ell-r_I-s_I}\times (\mathbb S^1)^{s_I}.
		\end{equation}
		In particular:
		\begin{enumerate}[label=\textup{(\roman*)},leftmargin=2.8em]
			\item the angular dimension of the leaf is $\ell-r_I$;
			\item $L$ is compact if and only if $V$ is rational with respect to the lattice $\Lambda_I$;
			\item in general, the closure $\overline L$ is the translated subtorus described in \Cref{prop:closure-subtorus}; if $V$ is not rational, $L$ is a non-closed immersed subgroup, dense in this subtorus.
		\end{enumerate}
	\end{theorem}
	
	\begin{proof}
		On the universal cover $\R^\ell\to T^\ell$, the angular foliation is constant and its leaves are the affine subspaces parallel to $V$. Fix one such affine subspace $x_0+V$. Two points of $x_0+V$ have the same image in $T^\ell$ exactly when they differ by an element of the lattice $\Lambda_I=(2\pi)\Z^\ell$. Hence the projected leaf is
		\[
		L\cong (x_0+V)/(V\cap \Lambda_I)\cong V/(V\cap \Lambda_I).
		\]
		Since $\dim_{\R}V=\ell-r_I$, the quotient of the real vector space $V$ by the free abelian subgroup $V\cap \Lambda_I$ has the standard form
		\[
		V/(V\cap \Lambda_I)\cong \R^{\ell-r_I-s_I}\times (\mathbb S^1)^{s_I},
		\]
		where $s_I=\rk_{\Z}(V\cap \Lambda_I)$. This proves \eqref{eq:leaf-topology-boundary} and the formula for the angular dimension.
		The compactness criterion is equivalent to the condition that $V\cap \Lambda_I$ be a full lattice in $V$, namely $s_I=\ell-r_I$, which is exactly the rationality of $V$ with respect to $\Lambda_I$. The final statement is the content of \Cref{prop:closure-subtorus}.
	\end{proof}

	The closure of an angular leaf is therefore the subtorus determined by the smallest rational subspace containing $V$. An irrational angular leaf need not be dense in the whole boundary torus.

	\section{Linear monodromy on \texorpdfstring{$X^{\log}$}{Xlog} and boundary limits}
	\label{sec:Flog}
	
	We study the local multiplicative coordinate near the logarithmic boundary and the extension of the monodromy local system to $X^{\log}$. Radial limits of the coordinate and extension of a local system are different questions; both can be compared with Kummer base change.
	\subsection{The local logarithmic formula}
	
	\begin{definition}\label[definition]{def:logcover}
		Let $\rho:\pi_1(U)\to \Gamma$ be the multiplicative holonomy character. Let $\pi: P_\rho\to U$ be the covering corresponding to
		$\ker(\rho)$. Via the canonical identification
		$
		\pi^{-1}(U)\cong U,
		$
		we also regard it as a covering of the interior of $X^{\log}$. When convenient, we denote this interior logarithmic covering by
		$
		\widetilde U^{\log}\longrightarrow \pi^{-1}(U)\subset X^{\log}.
		$
	\end{definition}
	
	For the construction below, one works over the interior $\pi^{-1}(U)\cong U$, while $X^{\log}$ supplies the angular boundary variables. The only rôle of the covering in \Cref{def:logcover} is to kill the multiplicative monodromy measured by $\rho$.
	Fix an snc chart in which $D$ is given by $z_1\cdots z_\ell=0$, and write
	\[
	\eta=\sum_{i=1}^{\ell}\alpha_i\,\frac{dz_i}{z_i}+\beta,
	\qquad
	d\beta=0.
	\]
	Let $W\subset X^{\log}$ be a simply connected open subset contained in this chart. Choose a primitive $h:W\to\mathbb C$ of $\pi^*\beta$, so that $dh=\pi^*\beta$. On $W\cap \pi^{-1}(U)$ define
	\begin{equation}\label{eq:zlog-local}
		z_W^{\log}
		:=
		\exp(h)\cdot \prod_{i=1}^{\ell} r_i^{\alpha_i}\cdot
		\exp\!\left(i\sum_{i=1}^{\ell}\alpha_i\theta_i\right),
		\qquad
		r_i^{\alpha_i}:=\exp(\alpha_i\log r_i).
	\end{equation}
	
	\begin{lemma}\label[lemma]{lem:zlog-local}
		The function $z_W^{\log}$ satisfies
		\begin{equation}\label{eq:zlog-differential}
			\frac{dz_W^{\log}}{z_W^{\log}}=\pi^*\eta\qquad\text{on }W\cap\pi^{-1}(U).
		\end{equation}
	\end{lemma}
	
	\begin{proof}
		Taking logarithmic derivatives in \eqref{eq:zlog-local} gives
		\[
		d\log z_W^{\log}
		=
		dh+\sum_{i=1}^{\ell}\alpha_i\,d\log r_i
		+i\sum_{i=1}^{\ell}\alpha_i\,d\theta_i.
		\]
		Since $dh=\pi^*\beta$ and $d\log r_i+i\,d\theta_i=\pi^*(dz_i/z_i)$ by \eqref{eq:dz-over-z-kn}, this becomes
		\[
		d\log z_W^{\log}
		=
		\pi^*\beta+\sum_{i=1}^{\ell}\alpha_i\,\pi^*\!\left(\frac{dz_i}{z_i}\right)
		=
		\pi^*\eta,
		\]
		which is exactly \eqref{eq:zlog-differential}.
	\end{proof}
	
	Changing an angular lift by $2\pi n_i$ multiplies the local coordinate by $\exp(2\pi i n_i\alpha_i)$. Thus changes of angular lifts contribute residue multipliers, all of which belong to $\Gamma$.
	
	\begin{theorem}\label{thm:Flog}
		Let $D$ be a simple normal crossings divisor in a complex manifold $X$ and let $U=X\setminus D$. The character $\rho:\pi_1(U)\to\Gamma$ determines a locally constant principal $\Gamma$-bundle on the Kato--Nakayama space $X^{\log}$, extending the local system of the linear monodromy on $U$. Equivalently, it determines a classifying morphism of topological stacks
		\[
		b_\rho^{\log}:X^{\log}\longrightarrow B\Gamma.
		\]
		On a boundary torus over a stratum where the components $D_i$, $i\in I$, meet, its character is
		\[
		(m_i)_{i\in I}\longmapsto\exp\!\left(2\pi i\sum_{i\in I}m_i\alpha_i\right),\qquad
		\alpha_i=\Res_{D_i}(\eta).
		\]
		This statement concerns the local system and does not assert a continuous extension of the multiplicative developing coordinate to $\C^*$ at the boundary.
	\end{theorem}
	
	\begin{proof}
		For a simple normal crossings divisor, $X^{\log}$ is the real oriented blow-up of $D$, a manifold with corners whose interior is $U$. The inclusion $U\hookrightarrow X^{\log}$ is a homotopy equivalence. Hence restriction induces an equivalence between locally constant $\Gamma$-bundles on $X^{\log}$ and on $U$. Extend the bundle classified by $\rho$ and take its classifying map. In local polar coordinates $z_i=r_i e^{i\theta_i}$, a positive turn of $\theta_i$ represents a meridian of $D_i$, whose value under $\rho$ is $\exp(2\pi i\alpha_i)$. Multiplication gives the displayed character.
	\end{proof}
	
	\begin{corollary}\label[corollary]{cor:landingP1}
		For $r_i>0$, the local expression in Lemma~\ref{lem:zlog-local} satisfies
		\begin{equation*}
			\log|z_W^{\log}|=\Re(h)+\sum_{i=1}^\ell\bigl(\Re(\alpha_i)\log r_i-\Im(\alpha_i)\theta_i\bigr).
		\end{equation*}
		Along a boundary approach for which the right-hand side tends to $-\infty$, respectively $+\infty$, the local coordinate tends to $0$, respectively $\infty$, in $\PP^1$. No global extension to $[\PP^1/\Gamma]$ is claimed solely from these local limits.
	\end{corollary}
	
	\begin{proof}
		Take absolute values in the formula of Lemma~\ref{lem:zlog-local}. The conclusions follow from the usual topology of $\PP^1$. A global extension would in addition require compatible descent data across the boundary.
	\end{proof}
	
	\subsection{Boundary asymptotics, meridional monodromy, and Kummer base change}\label{subsec:boundary-asymptotics}
	
	The compactified target $[\PP^1/\Gamma]$ accommodates radial limits of the multiplicative coordinate but does not retain the full angular information at the boundary. The latter is carried by $X^{\log}$, together with the meridional characters.
	
	Let $D_i$ be an irreducible component of the simple normal crossings divisor $D$, and let
	$\eta\in \Omega_X^1(\log D)$be closed. Then the logarithmic residue 
	$\alpha_i:=\Res_{D_i}(\eta)$ is constant on $D_i$.
	For a fixed logarithmic connection form, the sign of $\Re(\alpha_i)$ controls the radial limit of the local multiplicative coordinate near a general point of $D_i$. This leads to the decomposition
	\[
	D=D^+\cup D^-\cup D^0,
	\]
	whose three parts record the sign of the radial exponent; angular monodromy is present in all three cases.
	
	\begin{definition}\label[definition]{def:radial-sign-decomposition}
		For each irreducible component $D_i$ of $D$, let
		$
		\alpha_i:=\Res_{D_i}(\eta)\in \C.
		$
		Set
		\[
		D^+ := \bigcup_{\Re(\alpha_i)>0} D_i,
		\qquad
		D^- := \bigcup_{\Re(\alpha_i)<0} D_i,
		\qquad
		D^0 := \bigcup_{\Re(\alpha_i)=0} D_i.
		\]
		Thus $
		D=D^+\cup D^-\cup D^0 $
		as a partition of its irreducible components; the supports of these unions may intersect.
	\end{definition}
	
	\begin{proposition}\label[proposition]{prop:radial-asymptotics-component}
		Let $x\in D_i$ be a smooth point of a single irreducible component $D_i$, and choose a simply connected neighbourhood on which
		\[
		\eta=\alpha_i\,\frac{df_i}{f_i}+\beta,
		\qquad d\beta=0,
		\]
		with $f_i=0$ defining $D_i$.
		Choose a primitive $h$ of $\beta$, and define the local multiplicative developing coordinate
		$
		z:=e^h f_i^{\alpha_i}.$
		On a simply connected angular sector, write $f_i=r_i e^{i\theta_i}$, with a bounded lift $\theta_i$ of the argument. Then
		$
		\log|z| = \Re(h)+\Re(\alpha_i)\log r_i-\Im(\alpha_i)\theta_i.
		$
		In particular:
		\begin{enumerate}[label=(\alph*),leftmargin=2.2em]
			\item if $\Re(\alpha_i)>0$, then $z\to 0$ as $r_i\to 0$ within the sector;
			\item if $\Re(\alpha_i)<0$, then $z\to \infty$ as $r_i\to 0$ within the sector;
			\item if $\Re(\alpha_i)=0$, then $|z|$ remains bounded above and away from zero on the chosen sector.
		\end{enumerate}
	\end{proposition}
	
	\begin{proof}
		Since $dh=\beta$, the function $h$ is holomorphic on the chosen neighbourhood, hence bounded there.
		On a simply connected sector choose a continuous lift $\theta_i$ of the argument of $f_i$. Then
		\[
		\log|e^h f_i^{\alpha_i}|=\Re(h)+\Re(\alpha_i)\log r_i-\Im(\alpha_i)\theta_i.
		\]
		After shrinking the sector, the angle and $h$ remain bounded. Since $\log r_i\to-\infty$, the sign of $\Re(\alpha_i)$ gives the first two limits. When $\Re(\alpha_i)=0$, the boundedness of $h$ and $\theta_i$ bounds $|z|$ above and away from zero.
	\end{proof}
	
	The radial alternatives do not describe the leaves of $\fol$.
	For example, take $\eta=dx$ and $\omega=e^{-x}dy$ on $\C^2$.
	Then $d\omega=\omega\wedge\eta$, but the level sets of $e^x$
	are not those of the foliation $dy=0$.

	\begin{figure}[ht]
		\centering
		\begin{tikzpicture}[scale=1,line cap=round,line join=round]
			\begin{scope}[shift={(0,0)}]
				\draw[thin] (0,0) circle (1.65);
				\fill (0,0) circle (1.2pt);
				\node[fill=white,inner sep=1pt] at (0.16,0.14) {\small $x$};
				
				\node[align=center] at (0,2.95)
				{local transverse model\\ near $x\in D_i\subset D^{+}$};
				\node at (0,2.20) {$\Re(\alpha_i)>0$};
				\draw[thick,midarrow]
				(-1.00,0.95) .. controls (-0.40,1.05) and (-0.05,0.42) .. (0.00,0.02);
				\draw[thick,midarrow]
				(-1.05,0.00) .. controls (-0.92,0.48) and (-0.32,0.30) .. (0.00,0.01);
				\draw[thick,midarrow]
				(-0.75,-0.95) .. controls (-0.18,-0.95) and (0.12,-0.30) .. (0.02,-0.02);
				\draw[thick,midarrow]
				(0.00,-1.10) .. controls (0.42,-0.98) and (0.30,-0.22) .. (0.02,-0.02);
				\draw[thick,midarrow]
				(0.85,-0.85) .. controls (0.90,-0.18) and (0.32,0.10) .. (0.02,0.01);
				\draw[thick,midarrow]
				(1.05,0.00) .. controls (0.92,0.55) and (0.32,0.35) .. (0.02,0.02);
				\draw[thick,midarrow]
				(0.90,0.95) .. controls (0.30,1.02) and (0.02,0.40) .. (0.02,0.02);
				
				\node at (0,-2.05) {$\Delta_x$};
				\node[align=center] at (0,-2.78)
				{$|z|\to 0$};
			\end{scope}
			\begin{scope}[shift={(5.6,0)}]
				\draw[thin] (0,0) circle (1.65);
				\fill (0,0) circle (1.2pt);
				\node[fill=white,inner sep=1pt] at (0.16,0.14) {\small $x$};
				
				\node[align=center] at (0,2.95)
				{local transverse model\\ near $x\in D_i\subset D^{0}$};
				\node at (0,2.20) {$\Re(\alpha_i)=0$};
				\draw[thick,midarrow]
				(1.10,0.00) arc[start angle=0,end angle=330,radius=1.10];
				\draw[thick,midarrow]
				(0.78,0.00) arc[start angle=0,end angle=325,radius=0.78];
				\draw[thick,midarrow]
				(0.45,0.00) arc[start angle=0,end angle=315,radius=0.45];
				
				\node at (0,-2.05) {$\Delta_x$};
				\node[align=center] at (0,-2.78)
				{no forced radial limit};
			\end{scope}
			\begin{scope}[shift={(11.2,0)}]
				\draw[thin] (0,0) circle (1.65);
				\fill (0,0) circle (1.2pt);
				\node[fill=white,inner sep=1pt] at (0.16,0.14) {\small $x$};
				
				\node[align=center] at (0,2.95)
				{local transverse model\\ near $x\in D_i\subset D^{-}$};
				\node at (0,2.20) {$\Re(\alpha_i)<0$};
				\draw[thick,midarrow]
				(0.00,0.02) .. controls (-0.05,0.40) and (-0.40,1.05) .. (-1.00,1.15);
				\draw[thick,midarrow]
				(0.00,0.01) .. controls (-0.32,0.30) and (-0.92,0.48) .. (-1.08,0.00);
				\draw[thick,midarrow]
				(0.00,-0.02) .. controls (-0.12,-0.30) and (-0.55,-0.92) .. (-0.95,-1.10);
				\draw[thick,midarrow]
				(0.02,-0.02) .. controls (0.30,-0.22) and (0.42,-0.98) .. (0.00,-1.20);
				\draw[thick,midarrow]
				(0.02,0.01) .. controls (0.32,0.10) and (0.90,-0.18) .. (0.92,-0.95);
				\draw[thick,midarrow]
				(0.02,0.02) .. controls (0.32,0.35) and (0.92,0.55) .. (1.08,0.00);
				\draw[thick,midarrow]
				(0.02,0.02) .. controls (0.02,0.40) and (0.30,1.02) .. (0.95,1.15);
				
				\node at (0,-2.05) {$\Delta_x$};
				\node[align=center] at (0,-2.78)
				{$|z|\to\infty$};
			\end{scope}
			
		\end{tikzpicture}
		\caption{Schematic radial behaviour of the local multiplicative coordinate near a smooth point $x\in D_i\subset D$.
			In a small transverse disk $\Delta_x$, the divisor is represented by the point $\Delta_x\cap D_i=\{x\}$.
			Left: $|z|\to0$ along $D_i\subset D^{+}$.
			Middle: no forced radial limit along $D_i\subset D^0$.
			Right: $|z|\to\infty$ along $D_i\subset D^{-}$. The arrows indicate limiting behaviour of the modulus.}
		\label{fig:radial-sign-phase-portraits}
	\end{figure}
	
	\begin{proposition}\label[proposition]{prop:monodromy-obstruction-extension}
		Let $\Delta$ be a small transverse disk to a smooth point of $D_i$, with punctured disk $\Delta^*:=\Delta\setminus\{0\}$.
		Then the restriction of the stacky developing map to $\Delta^*$ is represented by a principal $\Gamma$-bundle with monodromy
		$
		g_i = \exp\bigl(2\pi i\,\alpha_i\bigr)=\exp\bigl(2\pi i\,\Res_{D_i}(\eta)\bigr)\in \Gamma.
		$
		If $g_i\neq 1$, this $\Gamma$-bundle does not extend to $\Delta$. Consequently the composite
		\[
		\overline F_U=j\circ F:\Delta^*\longrightarrow [\PP^1/\Gamma]
		\]
		does not extend to a morphism $\Delta\to[\PP^1/\Gamma]$ of analytic stacks. Passing from the open to the compactified quotient does not remove the obstruction, because both quotients carry their natural morphisms to $B\Gamma$.
	\end{proposition}
	
	\begin{proof}
		The punctured disk $\Delta^*$ has fundamental group generated by a meridian around $0$.
		By definition of the holonomy character,
		\[
		\rho(\text{meridian}) = \exp\left(\int_{\text{meridian}}\eta\right)=\exp(2\pi i\alpha_i)=g_i.
		\]
		Hence the restricted stacky developing map corresponds to a principal $\Gamma$-bundle on $\Delta^*$ with monodromy $g_i$.
		If this bundle extended to a principal $\Gamma$-bundle on the whole disk $\Delta$, then, because $\Delta$ is simply connected, the extended bundle would be trivial.
		Its restriction to $\Delta^*$ would therefore have trivial monodromy.
		Thus extension of the principal bundle is impossible when $g_i\neq 1$. If $\overline F_U$ extended to $\Delta$, composing the extension with the canonical projection $[\PP^1/\Gamma]\to B\Gamma$ would provide precisely such an extension of the principal bundle, a contradiction.
	\end{proof}
	
	\begin{theorem}\label{thm:kummer-vs-xlog-boundary}
		Let $D_i$ be an irreducible component of $D$, let
		$
		\alpha_i:=\Res_{D_i}(\eta)\in\mathbb C,
		$ and $
		g_i:=\exp(2\pi i\alpha_i)\in\Gamma,
		$
		and consider the boundary behaviour of the developing coordinate near $D_i$. Then two complementary mechanisms occur.
		
		\begin{enumerate}[label=\textup{(\alph*)},leftmargin=2.4em]
			\item If $g_i$ has finite order $r$, then near a point
			of $D_i$ lying on no other component of $D$, the Kummer
			base change $f_i=t^r$ trivialises the meridional monodromy.
			The pulled-back multiplicative coordinate extends
			holomorphically to $\mathbb P^1$ across the lifted transverse
			disk, with value $0$, $\infty$, or an element of $\mathbb C^*$
			according as $\alpha_i>0$, $\alpha_i<0$, or $\alpha_i=0$.
			
			\item
			Assume that $\Re(\alpha_i)=0$. Then the radial analysis does not force a limit to $0$ or to $\infty$. Instead, the boundary behaviour is carried by the angular variable on $X^{\log}$. In local polar coordinates
			$
			f_i=r_i e^{i\theta_i},
			$
			one has
			\[
			\frac{df_i}{f_i}=\frac{dr_i}{r_i}+i\,d\theta_i,
			\]
			and on the universal cover of the boundary circle the horizontal multiplicative coordinate carries the factor
			$
			\exp(i\alpha_i\theta_i).
			$
			Equivalently, the corresponding boundary character has deck-transformation multiplier
			$
			\exp(2\pi i\alpha_i).
			$
			Thus $X^{\log}$ records the dependence on the angular variable, which is not visible at the level of a purely radial compactification.
		\end{enumerate}
	\end{theorem}
	
	\begin{proof}
		Assertion \textup{(a)} is exactly the local Kummer trivialisation described in \Cref{prop:kummer-trivialises}: after the root construction $f_i=t^r$, the residue becomes $r\alpha_i$ and the meridional multiplier becomes $
		\exp(2\pi i\,r\alpha_i)=g_i^r=1.
		$
		Since $g_i$ has order $r$, the relation $\exp(2\pi i r\alpha_i)=1$ forces $r\alpha_i\in\Z$, so $\alpha_i$ is real. Near a smooth point write $\eta=\alpha_i\,df_i/f_i+dh$ with $h$ holomorphic. A local multiplicative coordinate is $e^h f_i^{\alpha_i}$. After $f_i=t^r$ it becomes $e^{h(t^r,\ldots)}t^{r\alpha_i}$, which is meromorphic across $t=0$ and hence defines a holomorphic map to $\PP^1$ on the transverse disk. Its limiting value follows from the sign of the integer $r\alpha_i$. At a crossing, products with both positive and negative integral exponents may have indeterminacies, which is why the assertion is restricted to a smooth point.
		
		Assertion \textup{(b)} follows from \Cref{prop:boundary-fiber-restriction,cor:residue-monodromy}: when $\Re(\alpha_i)=0$, the radial analysis does not force a limit to $0$ or $\infty$, whereas the boundary fibre carries the character with multiplier
		$
		\exp(2\pi i\,\alpha_i).
		$
		Thus the Kato--Nakayama boundary records the angular part of the same residue data.
		Local Kummer base change kills finite meridional monodromy; $X^{\log}$ retains the angular boundary information before that change.
	\end{proof}
	
	\subsection{Logarithmic monodromy and the Betti local system}\label{sec:log-RH}
	Let $(\omega,\eta)$ be transversely affine data on $U=X\setminus D$, with $\eta$ a global logarithmic $1$-form on $X$, that is, $\eta\in H^0(X,\Omega_X^1(\log D))$, and $d\eta=0$. We interpret the boundary monodromy in the setting of the logarithmic Riemann--Hilbert correspondence \cite{Deligne70,Ogus03}. The closed logarithmic form $\eta$ defines the rank-one logarithmic flat connection
	\begin{equation*}
		\nabla_\eta:=d-\eta
	\end{equation*}
	on the trivial line bundle over $U$. On $W\cap\pi^{-1}(U)$, for a simply connected logarithmic chart $W\subset X^{\log}$, the function $z_W^{\log}$ defined in \eqref{eq:zlog-local} is a local horizontal section of $\nabla_\eta$. 
	
	The usual residue--monodromy formula takes the following form.
	
	\begin{corollary}\label[corollary]{cor:residue-monodromy}
		Near a smooth point of a component of $D$, choose a local transverse parameter $t$ and write
		\begin{equation*}
			\eta=\alpha\,\frac{dt}{t}+\cdots .
		\end{equation*}
		On the universal cover of the punctured transverse disk, write
		$
		t=re^{i\theta},\ w=\log r+i\theta,
		$
		so that $t=e^w$ and $dt/t=dw$. A local horizontal multiplicative section therefore contains the factor $\exp(\alpha w)$. The meridian acts by $w\mapsto w+2\pi i$, so the corresponding multiplier is
		$
		\exp(2\pi i\alpha).
		$
		In particular, if $\alpha=\Res_{D_i}(\eta)$, then the meridional holonomy around $D_i$ is $\exp(2\pi i\alpha)$.
	\end{corollary}
	
	\begin{proof}
		The identity $dt/t=dw$ follows   from  $w=\log r+i\theta$, and 
		$
		t=e^w$. Substituting this into the local expression of the horizontal multiplicative section of $\nabla_{\eta}$, one sees that the only non-single-valued factor is $\exp(\alpha w)$. Under a full meridian, $w$ is translated by $2\pi i$, so the multiplier is $\exp(2\pi i\alpha)$. This is the standard residue--monodromy rule; compare \cite[\S I.3]{Deligne70}.
	\end{proof}
	
	For rational residues, the effect of a local root construction on meridional monodromy can be calculated explicitly.
	
	\begin{proposition}\label[proposition]{prop:rational}
		Assume that
		\begin{equation}\label{eq:rational-residues}
			\Res_{D_i}(\eta)\in \frac{1}{N}\Z
			\qquad
			\text{for every irreducible component }D_i\subset D.
		\end{equation}
		Then, for every meridian $\gamma_i$ around $D_i$, one has
		$
		\rho(\gamma_i)\in \mu_N.
		$
		In particular, the subgroup of $\Gamma$ generated by the meridional holonomies is finite.
	\end{proposition}
	
	\begin{proof}
		For each meridian $\gamma_i$, \Cref{cor:residue-monodromy} gives $\rho(\gamma_i)=\exp\!\bigl(2\pi i\,\Res_{D_i}(\eta)\bigr)$. Under the assumption \eqref{eq:rational-residues}, the exponent is an integral multiple of $2\pi i/N$. Hence $\rho(\gamma_i)^N=1$, and therefore $\rho(\gamma_i)\in\mu_N$.
	\end{proof}
	
	The resulting meridional monodromy is finite and becomes trivial after the corresponding local Kummer base change.
	
	\begin{proposition}\label[proposition]{prop:kummer-trivialises}
		Assume \eqref{eq:rational-residues}. After the local $N$-th root construction along $D$, the meridional monodromy of the pulled-back linear connection becomes trivial. Equivalently, the boundary characters of the pulled-back linear local system are trivial in the meridional directions.
	\end{proposition}
	
	\begin{proof}
		Work locally in an snc chart in which
		\begin{equation*}
			\eta=\sum_{i=1}^{\ell}\alpha_i\,\frac{dz_i}{z_i}+\beta,
			\qquad
			\alpha_i\in \frac{1}{N}\Z,
			\qquad
			d\beta=0.
		\end{equation*}
		Introduce root coordinates by
		$
		z_i=w_i^N.
		$
		Then
		\begin{equation*}
			\frac{dz_i}{z_i}=N\,\frac{dw_i}{w_i},
		\end{equation*}
		so the pulled-back logarithmic form becomes
		\begin{equation*}
			\eta'=\sum_{i=1}^{\ell}N\alpha_i\,\frac{dw_i}{w_i}+\beta'.
		\end{equation*}
		Since each $N\alpha_i$ is an integer, the residue of $\eta'$ along $w_i=0$ is integral. Therefore the corresponding meridional multiplier is $\exp(2\pi i\,N\alpha_i)=1$, and the meridional holonomy is trivial after Kummer base change.
		The same statement may be read directly on the logarithmic boundary. Before base change, on the universal cover of a boundary torus one has
		\begin{equation*}
			\widetilde\chi_I(\theta_1,\dots,\theta_\ell)
			=
			\exp\left(i\sum_{i=1}^{\ell}\alpha_i\theta_i\right).
		\end{equation*}
		Under the root construction, the angular variables rescale by
		$
		\theta_i=N\varphi_i,
		$
		so the transformed character becomes
		\begin{equation*}
			\widetilde\chi_I'(\varphi_1,\dots,\varphi_\ell)
			=
			\exp\!\left(i\sum_{i=1}^{\ell}N\alpha_i\varphi_i\right).
		\end{equation*}
		Because each $N\alpha_i$ is an integer, a full turn in any meridional direction acts trivially. Thus the meridional character becomes trivial after the Kummer base change.
	\end{proof}
	
	\begin{corollary}\label[corollary]{cor:kummer-trivialisation}
		Assume that all residues of $\eta$ are rational. Then, after a suitable local Kummer base change along $D$, the linear part of the  transversely affine structure has trivial meridional monodromy. Equivalently, the rank-one Betti local system attached to $\nabla_\eta$ becomes meridionally trivial, and the boundary characters become trivial on the corresponding logarithmic boundary circles.
	\end{corollary}
	
	\begin{proof}
		Fix a point of $D$ and choose a simple normal crossings chart about it. Only finitely many components meet this chart, so their rational residues have a common denominator $N$. Applying \Cref{prop:kummer-trivialises} in the chart kills the meridional monodromy there. The denominators need not admit a uniform bound on a non-compact $X$; the assertion is local.
	\end{proof}
	
	Thus local Kummer base change trivialises finite meridional
	monodromy, while the Kato--Nakayama space records the angular
	boundary data before that base change.
	
	\medskip
	
	\medskip\noindent\textbf{Acknowledgements.}
	PPB acknowledges support from the Coordena\c{c}\~ao de Aperfei\c{c}oamento de Pessoal de N\'ivel Superior (CAPES), Finance Code 001, and the PDSE programme, grant no.~88881.220396/\allowbreak 2025-01. GB acknowledges support from CNPq. Both thank the University of Bari for its hospitality. MC thanks Javier Gargiulo Acea, Pablo Perrella and  Sebasti\'an Velazquez for discussions on logarithmic geometry. He is partially supported by the Universit\`a degli Studi di Bari and is a member of INdAM--GNSAGA. 
	
	\subsection*{Usage of Large Language Models}
	The mathematical content, ideas, constructions and results of this paper are the work of the authors.  ChatGPT  was used for linguistic and editorial revision, checking local calculations, suggesting references and preparing figures. The authors take full responsibility for the content of the manuscript.

\end{document}